\documentclass[pdflatex,sn-mathphys-num]{sn-jnl}
\usepackage{lscape}
\usepackage{amssymb, palatino, geometry,longtable}
\usepackage{amsthm, amsfonts, amsmath, latexsym}
\usepackage{epsfig,graphicx,xcolor}
\usepackage{epstopdf}
\usepackage{subfigure}

\usepackage{algorithm}%
\usepackage{algorithmicx}%
\usepackage{algpseudocode}%
\usepackage{enumitem}
\usepackage{booktabs}

\usepackage{url}                
\usepackage{version}        
\usepackage{multirow}
\usepackage{color}
\usepackage{dsfont}
\newcounter{thm}[section]

\newtheorem{theorem}[thm]{Theorem}

\newcounter{dfn}[section]

\newcounter{rmk}[section]

\def\theequation{\thesection.\arabic{equation}}

\newcommand{\normmm}[1]{{\left\vert\kern-0.25ex\left\vert\kern-0.25ex\left\vert #1
   \right\vert\kern-0.25ex\right\vert\kern-0.25ex\right\vert}}

\newcommand{\bfA}[1]{{\textbf A}(#1)}

\newcommand{\jw}[1]{{{\color{purple} (JW: #1)}}}

\DeclareMathOperator*{\argmin}{arg\,min}

\usepackage{graphicx} 
\title{Numerical Optimization for Tensor Disentanglement}

\begin{document}
\author[1]{\fnm{Julia} \sur{Wei}}\email{juliawei@g.harvard.edu}

\author[2]{\fnm{Alec} \sur{Dektor}}\email{adektor@lbl.gov}

\author[3]{\fnm{Chungen} \sur{Shen}}\email{shenchungen@usst.edu.cn}

\author[4]{\fnm{Zaiwen} \sur{Wen}}\email{wenzw@pku.edu.cn}

\author*[2]{\fnm{Chao} \sur{Yang}}\email{CYang@lbl.gov}

\affil[1]{\orgdiv{Department of Physics}, \orgname{Harvard University}, \orgaddress{\street{17 Oxford St.}, \city{Cambridge}, \state{MA}  \postcode{02138}, \country{USA}}}

\affil[2]{\orgdiv{Applied Mathematics and Computational Research Division}, \orgname{Lawrence Berkeley National Laboratory}, \orgaddress{\street{1 Cyclotron Rd}, \city{Berkeley}, \state{CA} \postcode{94720}, \country{USA}}}

\affil[3]{\orgdiv{College of Science}, \orgname{University of Shanghai for Science and Technology}, \orgaddress{\street{516 Jungong Rd}, \city{Shanghai}, \postcode{200093}, 
\country{China}}}

\affil[4]{\orgdiv{Beijing International Center for Mathematical Research}, \orgname{Peking University}, \orgaddress{\street{78  Jingchunyuan}, \city{Beijing}, \postcode{100871}, 
\country{China}}}


\abstract{ 
Tensor networks provide compact and scalable representations of high-dimensional data, enabling efficient computation in fields such as quantum physics, numerical partial differential equations (PDEs), and machine learning.
This paper focuses on tensor disentangling, the task of identifying transformations that reduce bond dimensions by exploiting gauge freedom in the network.
We formulate this task as an optimization problem over orthogonal matrices acting on a single tensor's indices, aiming to minimize the rank of its matricized form. We present Riemannian optimization methods and a joint optimization framework that alternates between optimizing the orthogonal transformation for a fixed low-rank approximation and optimizing the low-rank approximation for a fixed orthogonal transformation, offering a competitive alternative when the target rank is known. To seek the often unknown optimal rank, we introduce a binary search strategy integrated with the disentangling procedure. Numerical experiments on random tensors and tensors in an approximate isometric tensor network state are performed to compare different optimization methods and explore the possibility of combining different methods in a hybrid approach.
}


\maketitle

\section{Introduction}
\label{sec:intro}
Tensor networks are widely used in computational science to efficiently represent high-dimensional data and operators, enabling scalable algorithms for problems that would otherwise be intractable. A tensor network is a structured collection of multi-dimensional arrays (tensors) connected by products over shared indices. Prominent examples of tensor network structures include matrix product states (MPS)~\cite{Perez2006} or equivalently tensor trains (TT)~\cite{Oseledets2011}, projected entangled pair states (PEPS)~\cite{Orus2014, Cirac2021}, the multi-entanglement renormalization ansatz (MERA)~\cite{vidalClassQuantumManyBody2008, Evenbly2015}, as well as the canonical polyadic (CP), Tucker~\cite{kolda2009tensor}, and Hierarchical Tucker decompositions~\cite{Uschmajew2013}. Tensor networks have been used as compact representations of large-scale data in several applications, e.g., quantum many-body systems \cite{Orus2019}, numerical methods for partial differential equations (PDEs) \cite{Bachmayr2016, Dektor2021}, data analysis \cite{cichocki2014}, and machine learning \cite{Stoudenmire2018, Zangrando2024}. 

A convenient diagrammatic notation for describing algorithms involving tensor networks is shown in Figure~\ref{fig:TNexample}~\cite{SCHOLLWOCK201196,ciracMatrixProductStates2021}. In these diagrams, each circle represents a tensor, and each connecting line denotes an index over which the two tensors are contracted. These connecting lines are called internal legs or bonds, and they encode the internal structure of the network. Lines that connect to only one tensor are called open legs, and they correspond to the indices of the overall tensor represented by the network. The number of lines emanating from a tensor indicates its dimension, and the number of open legs in the network indicates the dimension of the resulting global tensor. The size of each internal bond, i.e., the number of values it can take, is called the bond dimension, and it plays a key role in determining the representational power and computational cost of the network. 

\begin{figure}[!hbtp]
    \centering    \includegraphics[width=0.8\textwidth]{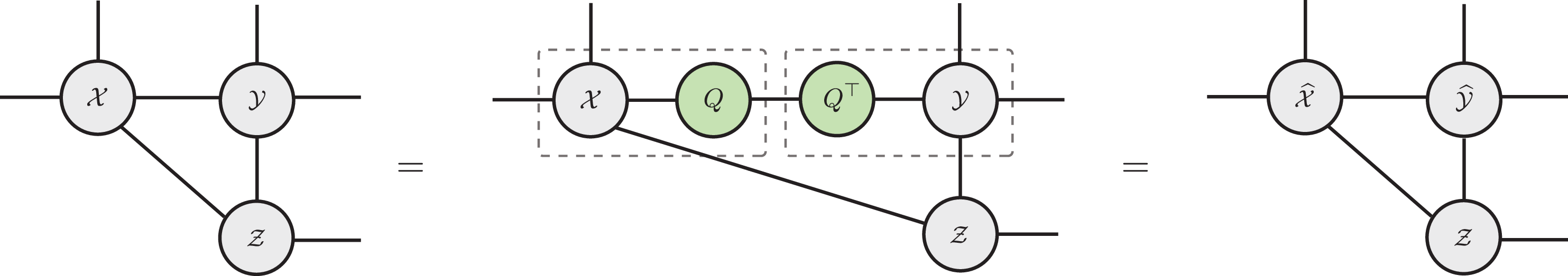}
    \caption{A tensor network consisting of three tensors $\mathcal{X}, \mathcal{Y}, \mathcal{Z}$. Applying a gauge transformation $Q$ to the tensor network results in a new, equivalent tensor network consisting of $\widehat{\mathcal{X}}, \widehat{\mathcal{Y}}, {\mathcal{Z}}$. }
    \label{fig:TNexample} 
\end{figure}

Many tensor networks have internal degrees of freedom that can be optimized to reduce the size of certain bonds, leading to a more efficient representation. When this is possible, we say that the tensor network can be \textit{disentangled}. One example is inserting an invertible matrix and its inverse between two connected tensors. This process, often referred to as a gauge transformation, leaves the tensor network invariant, though it can change the bond dimension when loops are present in the network. 
Figure~\ref{fig:TNexample} illustrates a gauge transformation, where an orthogonal matrix $Q$ and its inverse $Q^{\top}$ are inserted between tensors $\mathcal{X}$ and $\mathcal{Y}$. This transformation yields an equivalent tensor network in which $\widehat{\mathcal{X}}$ and $\widehat{\mathcal{Y}}$ are obtained by contracting $\mathcal{X}$ with $Q$ and $\mathcal{Y}$ with $Q^{\top}$, respectively. 
In this case, it may be possible to choose a $Q$ so that $\widehat{\mathcal{X}}$ can be decomposed into the contraction of two tensors $\widehat{\mathcal{U}}$ and $\widehat{\mathcal{W}}$ with lower bond dimension, e.g., using a truncated singular value decomposition, as shown in Figure~\ref{fig:NewTensor}, where the dimension across which the SVD is performed is indicated by a red dashed line. Contracting $\widehat{\mathcal{W}}$ with $\mathcal{Z}$ to form $\widehat{\mathcal{Z}}$ results in an equivalent tensor network with a reduced bond dimension between $\widehat{\mathcal{U}}$ and $\widehat{\mathcal{Z}}$. 

Tensor disentanglers play a key role in a variety of tensor network algorithms. They are particularly important in classical simulations of quantum many-body systems, appearing in methods such as tensor network renormalization~\cite{evenblyAlgorithmsEntanglementRenormalization2009, Evenbly2015}, MERA~\cite{vidalClassQuantumManyBody2008, evenbly2014class}, purified MPS~\cite{karraschReducingNumericalEffort2013,hauschildFindingPurificationsMinimal2018}, and isometric tensor network states (isoTNS)~\cite{zaletel2020isometric, lin2022efficient, sappler2025diagonal}. 
Disentanglers can also be used to decompose global unitary operators into sequences of unitary gates~\cite{slagle2021fast}, enabling efficient implementation of matrix operations on quantum hardware, with applications in quantum machine learning~\cite{Aizpurua2024}. 
Related optimization problems arise when determining basis transformations that reduce entanglement of a quantum state, e.g., Gaussian mode transformations in quantum chemistry~\cite{Krumnow2016}, and determining coordinate transformations so that high-dimensional functions and operators admit low-rank tensor network representations~\cite{Dektor2023,Dektor2024}. 

\begin{figure}[!hbtp]
    \centering
    \includegraphics[width=0.8\textwidth]{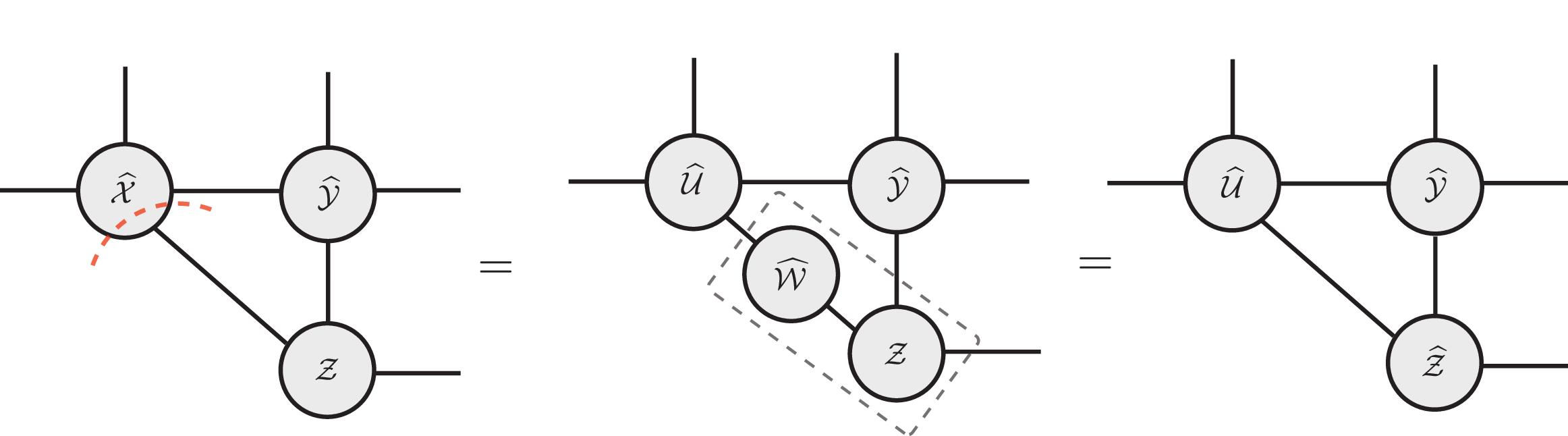}
    \caption{Performing a truncated SVD on $\hat{\mathcal X}$ across the dimension indicated by the red dashed line yields $\hat{{\mathcal U}}$ and $\hat{{\mathcal W}}$, which can have smaller dimensions than the bond connecting $\hat{\mathcal X}$ and ${\mathcal Z}$, provided the gauge transformation $Q$ shown in Figure~\ref{fig:TNexample} is chosen to disentangle $\hat{\cal{X}}$. Contracting $\hat{\mathcal W}$ and ${\mathcal Z}$ to obtain $\hat{{\mathcal Z}}$ yields a new tensor network with reduced bond dimension. \label{fig:NewTensor}}
\end{figure}

In this paper, we focus on a single tensor ${\cal X}$ within a tensor network, where bipartitioning the indices of ${\cal X}$ results in an unfolding matrix, $M$. We seek an orthogonal matrix $Q$ that, when contracted with $\mathcal{X}$ along a subset of indices, yields an $M(Q)$ with a low-rank approximation. Thus, $Q$ disentangles the system described by ${\cal X}$ along a particular bipartition.

We formulate the tensor disentangling task as a numerical optimization problem, highlighting the choice of different objective functions and incorporating the manifold constraint that restricts $Q$ to be an orthogonal matrix. The optimization problem can be solved by first- and second-order Riemannian optimization methods, which require the computation of Riemannian gradients and Hessian-vector products. We detail efficient strategies for evaluating these quantities and analyze their computational complexity, as they represent the primary cost of Riemannian optimization.

In addition, we reformulate the disentangling problem as a joint optimization over $Q$ and a low-rank approximation of $M(Q)$, which can be solved via an alternating minimization algorithm. Through numerical experiments on isoTNS tensors, a subclass of PEPS representing two-dimensional quantum states, we demonstrate that this approach can outperform Riemannian methods in terms of efficiency. However, it assumes prior knowledge of the optimal rank of $M(Q)$, which is typically unknown in practice. To address this, we propose methods for finding an initial estimate of the rank and optimizing it iteratively using a binary search algorithm integrated with the disentangling procedure. Finally, through numerical experiments on random tensors, we show that combining Riemannian optimization with alternating minimization yields an efficient hybrid algorithm.

This paper is organized as follows. In the next section, we give a precise formulation of the tensor disentangling task as an optimization problem and introduce the notations used to define the objective function and constraints.  Section~\ref{sec:unconstrained} details a specific parameterization of orthogonal matrices that allows the task to be cast as an unconstrained optimization problem. 
We present Riemannian optimization methods~\cite{hauruRiemannianOptimizationIsometric2021,luchnikovQGOptRiemannianOptimization2021a,lin2022efficient} for solving the disentangling optimization problem in section~\ref{sec:riemannian}. 
In section~\ref{sec:altopt}, we present the joint optimization framework and present the alternating minimization algorithm. We discuss some practical considerations in performing disentangling optimization in section~\ref{sec:practical}. Numerical examples are given in section~\ref{sec:examples} to demonstrate the effectiveness of the methods proposed in this work and compare the efficiency of different methods for random and isoTNS tensors.

\section{Problem Formulation and Notation}

We consider a $4$-dimensional tensor ${\cal X} \in \mathbb{R}^{l \times r \times b \times c}$. To describe the disentangling problem, we introduce a few operators. Let $\textbf{M}: \mathbb{R}^{l\times r\times b\times c} \to \mathbb{R}^{lr\times bc}$ be the map that flattens a four-dimensional tensor into a matrix by combining the first two dimensions into rows and the last two dimensions into columns. Throughout, we let $X=\textbf{M}({\cal X}) \in \mathbb{R}^{lr \times bc}$. We denote the reverse mapping of a matrix of dimension $lr \times bc$ to a tensor of dimension $l \times r \times b \times c$ by $\textbf{M}^{-1}$ so that ${\cal X} = \textbf{M}^{-1}(X)$. 
Let $\textbf{P}$ be a permutation that rearranges the dimensions of a tensor in the order $[1,4,2,3]$ so that $\textbf{P}({\cal X}) \in \mathbb{R}^{l\times c\times r\times b}$. We use $\textbf{P}^{-1}$ to denote the inverse of $\textbf{P}$ so that $\textbf{P}^{-1}(\textbf{P}({\cal X}))={\cal X}$. 
Finally, we define composition map  $\textbf{A}=\textbf{M}\circ \textbf{P}\circ \textbf{M}^{-1} : \mathbb{R}^{lr \times bc} \to \mathbb{R}^{lc \times rb}$, which is a unitary operator with inverse ${\bf A}^{-1}:\mathbb{R}^{lc \times rb} \to \mathbb{R}^{lr \times bc}$ defined as $\textbf{A}^{-1} \equiv \textbf{M}\circ \textbf{P}^{-1} \circ \textbf{M}^{-1}$. 
In Figure~\ref{fig:tensor_operators}, we provide graphical depictions of these operators. 
\begin{figure}[!hbtp]
    \centering
    \includegraphics[width=0.5\textwidth]{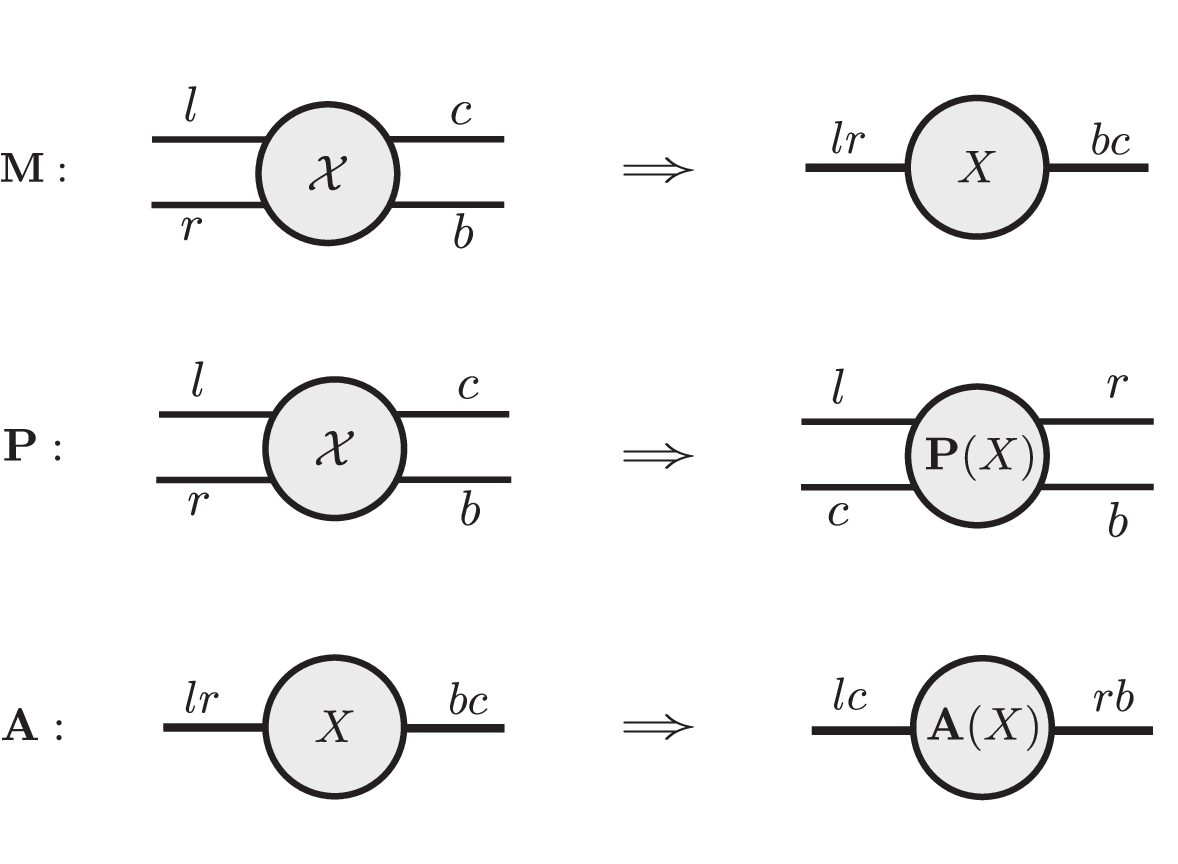}
    \caption{\label{fig:tensor_operators} Tensor reshaping and permutation operations used to define the disentangler optimization problem. }
\end{figure}
The optimization algorithms described in the following sections can be trivially extended to higher-dimensional tensors by defining alternative reshaping and permutation operators.

We are interested in finding an orthogonal matrix $Q \in \mathbb{R}^{lr\times lr}$ that minimizes
\begin{equation} \label{eq:sv_tail}
c_k(Q) = \sum_{i=k+1}^{m} \sigma_{i}^2(\bfA{QX})
\end{equation}
for some $k>1$ and $k \ll m$. Here $m = \min (lc,br)$ and the function $\sigma_i(\bullet)$ returns the $i$-th singular value of its input. The function in \eqref{eq:sv_tail} measures the error of replacing $\bfA{QX}$ with its best rank-$k$ approximation. Equivalently, \eqref{eq:sv_tail} measures the entanglement across a bond dimension of ${\cal X}$. When both \eqref{eq:sv_tail} and $k$ are small, the matrix $\bfA{QX}$ can be replaced by a low-rank approximation, thereby reducing the total number of degrees of freedom in the tensor network representation of ${\cal X}$. In the physics literature, this procedure is often referred to as \textit{disentanglement}, and $Q$ is referred to as a \textit{disentangler}.
A tensor network diagram of the disentangler problem is shown in Figure~\ref{fig:disentangler}. 
\begin{figure}[!hbtp]
    \centering
    \includegraphics[width=0.5\textwidth]{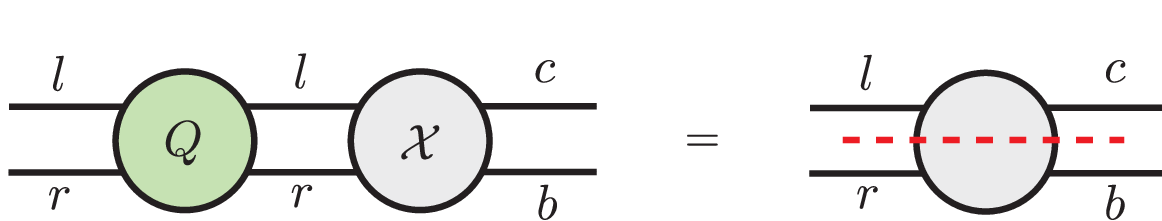}
    \caption{\label{fig:disentangler} Tensor network diagram of the disentangler problem. The unitary operator $Q$ is to be optimized so that entanglement across the red dashed line is minimized.}
\end{figure}
Constraining $Q$ to be orthogonal ensures that $\textbf{A}(QX)$ has the same Frobenius norm as that of $\bfA{X}$, i.e., 
\begin{equation}
\sum_{i=1}^{m} \sigma_i^2(\textbf{A} (Q X)) = \sum_{i=1}^{m} \sigma_i^2(\bfA{X}). 
\label{eq:sigma2sum}
\end{equation} 
Such a constraint is important so that \eqref{eq:sv_tail} is minimized relative to the norm of the tensor ${\cal X}$ that is being disentangled. 
More broadly, we can consider objectives of the form 

\begin{equation} \label{eq:g_def}
    f(Q) = \sum_{i=1}^m \phi\left(\sigma_i(\bfA{QX})\right),
\end{equation}
which depend on the choice of function $\phi: \mathbb{R} \to \mathbb{R}$. We can then formulate the disentangling of ${\cal X}$ as a constrained optimization problem 
\begin{equation} \label{eq:opt_constrained} 
\min_{Q^{\top}Q = I} \sum_{i=1}^m \phi\left(\sigma_i(\bfA{QX})\right).
\end{equation}

If $\phi(t) = \phi_{\mu}(t)$ is the piecewise function 
\begin{equation} \label{eq:Heaviside}
\phi_{\mu}(t) = \left \{
\begin{array}{ll}
0 & \mbox{if $t \leq \mu$}, \\
t^2 & \mbox{if $t  > \mu$},
\end{array}
\right.
\end{equation}
for some thresholding parameter $\mu$ satisfying $\sigma_k < \mu < \sigma_{k+1}$, then the objective function in \eqref{eq:opt_constrained} is equivalent to \eqref{eq:sv_tail}. In this case, the disentangler $Q$ is chosen to minimize the error in the rank-$k$ approximation of $\bfA{QX}$.

When the parameter $\mu$ (or equivalently, the target disentanglement rank $k$) is not known a priori, other objective functions can be used to disentangle ${\cal X}$. One such function is the von Neumann entropy~\cite{nielsenQuantumComputationQuantum2010}, 
\begin{equation}
S_\textrm{vN} \equiv - \sum_{i=1}^m \sigma_i^2 \ln \sigma_i^2, \label{eqn:vonneumann}
\end{equation}
obtained by setting $\phi(t)$ to 
$\phi(t) = -t^2\ln t^2$.

Another possibility is to choose $\phi(t)= t^{2\alpha}$ for some $\alpha \in (0,1)$. Minimizing the objective function \eqref{eq:g_def} obtained from such a $\phi(t)$ is equivalent to minimizing the Rényi-$\alpha$ entropy~\cite{verstraeteMatrixProductStates2006}, defined by
\begin{equation}
S^\alpha \equiv \frac{1}{1-\alpha}\ln(\sum_{i=1}^m \sigma_i^{2\alpha}).
\label{eqn:renyi}
\end{equation}
When $\alpha=1/2$, minimizing $S^{1/2}$ is equivalent to minimizing $\sum_{i=1}^m \sigma_i$, which is the nuclear norm of $\textbf{A}(QX)$, a well-known surrogate objective for the rank-minimization problem \cite{Recht2010, Lu2015}. 

Although we focus on real valued tensor in this work, the algorithms to be presented below can be easily extended to complex valued tensors, and the disentangler can be a complex and unitary matrix.

\section{Unconstrained optimization}
\label{sec:unconstrained}
The constrained optimization problem \eqref{eq:opt_constrained} can be reformulated as an unconstrained optimization problem if we use a particular parameterization of the orthogonal matrix $Q$. We discuss methods for solving such an unconstrained optimization problem in this section. In particular, we show how the gradient required in a first-order numerical optimization method can be computed.

Any orthogonal matrix can be parameterized as a matrix exponential, $Q=e^B$, where $B$ is a skew-symmetric matrix, i.e., $B = -B^{\top}$. If we collect all matrix elements in the strictly lower triangular part of $B$ into a vector $s$, and use $B(s)$ to represent the corresponding skew-symmetric matrix, then \eqref{eq:opt_constrained} can be formulated as an unconstrained optimization problem.

Substituting $Q=e^{B(s)}$ into \eqref{eq:g_def} yields the objective function 
\begin{equation}
g(s) \equiv f(e^{B(s)}).
\end{equation}

It follows from the chain rule that the gradient of $g(s)$ can be evaluated as 
\begin{equation}
\nabla g(s) = \frac{\partial f(Q)}{\partial Q} \cdot \frac{\partial Q}{\partial s}.
\end{equation}

The derivative of $Q$ with respect to $s$ can be derived from a perturbative expansion of $e^{B+\Delta}$ for some $\Delta$. Because the directional derivative of $e^{B}$ along $\Delta$ contains nested commutators of $B$ and $\Delta$, the derivative does not have a simple and closed form, except in the case $s=0$ in which it is simply the identity matrix.

The derivative of $f(Q)$ with respect to $Q$ can be evaluated by a closed-form expression given in the following 
theorem, which can also be found in~\cite{lese:2005b}.
This expression makes use of the singular value decomposition of $\bfA{QX}$.

\begin{theorem} \label{thm:euclidean_grad} 
    If $\bfA{QX} = U(Q) \Sigma(Q) V(Q)^{\top}$ is the singular value decomposition, then the gradient of $f$ in \eqref{eq:g_def} with respect to $Q$ is 
    \begin{equation} \label{eq:euclidean_grad}
    \nabla_Q f(Q) = {\bf A}^{-1}\left(U\phi'(\Sigma) V^{\top} \right) X^{\top}, 
    \end{equation}
    where the derivative $\phi'$ of $\phi$ is applied entry-wise to $\Sigma$. 
\end{theorem}
\begin{proof}
   We begin by taking the directional derivative of \eqref{eq:g_def} in the direction $E \in \mathbb{R}^{lr \times lr}$ and applying the chain rule to obtain 
   \begin{equation} \label{eq:g_diff}
   \begin{aligned}
       Df(Q)[E] &= \sum_{i=1}^m \phi'\left(\Sigma_{ii}\right) D\Sigma_{ii}(Q)[E] \\
       &= \mbox{trace}\left(\phi'(\Sigma) D \Sigma(Q)[E]\right), 
   \end{aligned}
   \end{equation}
   where the second equality holds since $\phi'(\Sigma)$ is diagonal. 
   Next, we obtain an expression for the directional derivative $D\Sigma(Q)[E]$. To do so, we differentiate $\bfA{QX}$ 
   \begin{equation}
   \begin{aligned}
       D \bfA{QX}[E] &= \bfA{EX} \\
       &= \left(DU(Q)[E]\right) \Sigma V^{\top} + U \left(D\Sigma(Q)[E]\right) V^{\top} + U\Sigma \left(DV(Q)^{\top}[E]\right),
    \end{aligned}
   \end{equation}
   where we used the fact that ${\bf A}$ is linear to obtain the first equality and the SVD representation of $\bfA{QX}$ with the product rule to obtain the second equality. Multiplying the preceding equation on the left by $U^{\top}$ and on the right by $V$ we obtain 
   \begin{equation} \label{eq:sigma_diff}
    D \Sigma(Q)[E] = U^{\top} \bfA{EX} V,
   \end{equation}
   where we used the fact that $U$ and $V$ have orthonormal columns and that $U^{\top}\left(DU(Q)[E]\right) =0$ and $\left(DV(Q)^{\top}[E]\right)V=0$. 
   Substituting \eqref{eq:sigma_diff} into \eqref{eq:g_diff} we obtain 
   \begin{equation} \label{eq:cyclic_trace}
   \begin{aligned}
       Df(Q)[E] &= \mbox{trace}\left(\phi'(\Sigma) U^{\top} \bfA{EX}V\right) \\
       &= \mbox{trace}\left(V \phi'(\Sigma) U^{\top} \bfA{EX}\right) \\
       &= \left\langle U\phi'(\Sigma)V^{\top}, \bfA{EX} \right\rangle \\
       &= \left\langle {\bf A}^{-1}\left( U\phi'(\Sigma)V^{\top}\right), EX \right\rangle \\
       &= \left\langle {\bf A}^{-1}\left( U\phi'(\Sigma)V^{\top}\right)X^{\top}, E \right\rangle
    \end{aligned}
   \end{equation}
   where we used the matrix inner product defined as $\langle X,Y\rangle=\mbox{trace}(X^{\top}Y)$ for any matrices $X,Y$ of the same dimension. The gradient $\nabla_Q f(Q)$ is the unique matrix satisfying $Df(Q)[E] = \langle \nabla_Q f(Q),E\rangle$ for all matrices $E$ of appropriate dimension. Comparing this with the last line in \eqref{eq:cyclic_trace} completes the proof. 
\end{proof}

\noindent 
The following algorithm summarizes the main steps for computing such a gradient, which will be utilized in numerical optimization algorithms for tensor disentanglement described in the following sections. 
\begin{algorithm} \label{alg:Euclidean_grad}
 \caption{Compute Euclidean gradient of disentangle cost function} 
\begin{algorithmic}[1]
\Require 
\Statex $X \in \mathbb{R}^{lr \times bc}$ $\rightarrow$ flattening of tensor ${\cal X}$ to be disentangled
\Statex $Q \in O(lr)$ $\rightarrow$ disentangler at current iteration
\Ensure 
$\nabla_Q f(Q)$ $\rightarrow$ Euclidean gradient of $f$ at $Q$
\State $Y = \bfA{QX}$
\State $[U,\Sigma,V] = \texttt{svd}(Y)$ 
\State $\nabla_Q f(Q) = {\bf A}^{-1}(U\phi'(\Sigma)V^{\top})X^{\top}$ 
\end{algorithmic}
\end{algorithm}

\noindent
We estimate the computational complexity of Algorithm~\ref{alg:Euclidean_grad} assuming $lc \leq rb$ for simplicity. 
\begin{enumerate}[label={Step \arabic*:}, leftmargin=*]
\item Compute the matrix product $QX$, where $Q \in \mathbb{R}^{lr \times lr}$ and $X \in \mathbb{R}^{lr \times bc}$. This costs $\mathcal{O}((lr)^2 bc)$ FLOPs. The result is reshaped by the operator ${\bf A}$ into a matrix $Y \in \mathbb{R}^{lc \times rb}$.

\item Perform an SVD of $Y \in \mathbb{R}^{lc \times rb}$. Since $lc \leq rb$, the cost is $\mathcal{O}((lc)^2 rb)$ FLOPs.

\item Apply the function $\phi'$ to the diagonal entries of $\Sigma$, and compute $U \phi'(\Sigma) V^\top$, which costs $\mathcal{O}((lc)^2 rb)$ FLOPs. Then, apply ${\bf A}^{-1}$ to reshape the result back to $\mathbb{R}^{lr \times bc}$ and multiply with $X^\top \in \mathbb{R}^{bc \times lr}$, costing $\mathcal{O}((lr)^2 bc)$ FLOPs.
\end{enumerate}
Adding the costs from each step, the total computational complexity is
$\mathcal{O}((lr)^2 bc + (lc)^2 rb)$.

We should note that the objective function defined in \eqref{eq:opt_constrained} is a function of a L\"{o}wner operator~\cite{lowner}, which generates a matrix-valued
 function via applying a single-variable function to each of the singular values of a matrix. Properties of L\"{o}wner operators, including the differentiation of such operators, have been well studied in~\cite{dsst:2018,lese:2005a,lese:2005b}. The gradient formula given in \eqref{eq:euclidean_grad} can also be derived directly from the results shown in these works. In particular, the case in which the singular value $\sigma_k$ has a multiplicity greater than 1 is considered in~\cite{lese:2005b}. In that case, $f(Q)$ is not differentiable and a Clark subdifferential can be used in a non-smooth optimization algorithm. 

The Euclidean gradient expression given in \eqref{eq:euclidean_grad} can be used in a first-order unconstrained optimization method. Starting from $s=0$, we can use the gradient to update $s$. Using the updated $s$, we can update $X$ as $X \leftarrow QX$, where $Q=e^{B(s)}$. Applying $Q$ to $X$ is equivalent to resetting $s$ to 0. Hence, in the next unconstrained optimization step, we can use \eqref{eq:euclidean_grad} again to compute the gradient and follow the steepest descent direction. This type of optimization is more conveniently formulated as a Riemannian optimization problem as we describe in the next section.

\section{Riemannian optimization}
\label{sec:riemannian}
Let $lr = n$. If we focus on the objective function $f(Q)$ referred to in Theorem~\ref{thm:euclidean_grad} and restrict $Q$ to the embedded Riemannian submanifold $O(n)$ of $\mathbb{R}^{n \times n}$ consisting of all orthogonal matrices, we can formulate the disentangling problem \eqref{eq:opt_constrained} as a Riemannian optimization and solve it using first- and second-order Riemannian optimization methods \cite{Absil2009,boumal2023introduction}. The idea of such methods is to traverse the manifold until we find a (local) solution to the optimization problem \eqref{eq:opt_constrained}. To traverse the manifold, we need a notion of direction which is provided by the tangent spaces. Recall that at each point $Q \in O(n)$ the tangent space of the manifold $O(n)$ is 
\begin{equation} \label{eq:tangent_space}
T_Q O(n) = \{ Q\Omega \in \mathbb{R}^{n\times n}: \Omega \in \textrm{Skew}(n)\},
\end{equation}
where $\textrm{Skew}(n)$ represents the set of all $n\times n$ skew symmetric matrices~\cite{boumal2023introduction}. The tangent space \eqref{eq:tangent_space} is a vector space that linearizes the manifold $O(n)$ around $Q$. We consider $O(n)$ as an embedded Riemannian submanifold of $\mathbb{R}^{n \times n}$, which means that every tangent space inherits an inner product $(\cdot,\cdot)$ from $\mathbb{R}^{n \times n}$. The collection of all such inner products defines a Riemannian metric, which allows us to adopt the notions of gradients and Hessians to the manifold $O(n)$. Effective numerical optimization on Riemannian manifolds relies on the efficient computation of Riemannian gradients and Hessians. A straightforward method for obtaining Riemannian gradients and Hessians on $O(n)$ utilizes the orthogonal projector $\textrm{Proj}_Q:\mathbb{R}^{n\times n} \to T_Q O(n)$ onto the tangent space \eqref{eq:tangent_space}, defined by 
\begin{equation}
\textrm{Proj}_Q(Z) =\frac{Z - Q Z^{\top} Q}{2}.
\label{eq:orthoproj}
\end{equation}
Hereafter we derive expressions for the Riemannian gradient and Hessian of the disentangler cost function \eqref{eq:g_def} and discuss their computational costs. Tensor diagrams of these expressions, as well as a derivation for complex manifolds, can be found in~\cite{sappler2024}.

\subsection{Riemannian gradient} \label{sec:riemannian_grad}
The Riemannian gradient of $f$ at $Q$, denoted by $\textrm{grad} f(Q)$, belongs to the tangent space $T_Q O(n)$ and points in the direction of steepest ascent of the cost function $f$ on the manifold $O(n)$. To obtain $\textrm{grad}f(Q)$ we project the Euclidean gradient $\nabla_Q f(Q)$, defined by \eqref{eq:euclidean_grad} in Theorem \ref{thm:euclidean_grad}, orthogonally onto the tangent space \eqref{eq:tangent_space}. 
Applying such projection to $\nabla_Q f(Q)$ yields the Riemannian gradient 
\begin{equation} \label{eq:riemannian_grad} 
\begin{aligned}
\textrm{grad} f(Q) &= \textrm{Proj}_Q(\nabla_Q f(Q)) 
\end{aligned}
\end{equation}
To compute the Riemannian gradient we first compute the Euclidean gradient using Algorithm~\ref{alg:Euclidean_grad} and then apply the orthogonal projector \eqref{eq:orthoproj} onto the tangent space. With the Riemannian gradient available, we can solve the disentangler problem \eqref{eq:opt_constrained} using first-order optimization algorithms on the manifold $O(n)$, such as gradient descent. 

The computational cost of the Riemannian gradient can be broken into two parts. First, there is the cost of the Euclidean gradient, which scales as 
$\mathcal{O}((lr)^2bc + (lc)^2 rb)$ as discussed in Section~\ref{sec:unconstrained}. Second, there is the cost of the orthogonal projection onto the tangent space defined in \eqref{eq:riemannian_grad}. The dominant cost in the projection is matrix multiplication between matrices of size $lr \times lr$, which scales as $\mathcal{O}((lr)^3)$. Thus the total cost of computing the Riemannian gradient scales as $\mathcal{O}((lr)^3 + (lr)^2bc + (lc)^2 rb)$. 
To perform one iteration of gradient descent, one computes the Riemannian gradient, chooses a step-size, and then applies a retraction mapping that moves the current iterate on the manifold $O(n)$ in the descent direction. 

\subsection{Riemannian Hessian} 

The Riemannian Hessian of $f$ at $Q$, denoted by $\textrm{Hess}f(Q)$, is a symmetric linear map acting on the tangent space $T_QO(n)$ 
\begin{equation}
\textrm{Hess} f(Q): T_Q O(n) \rightarrow T_Q O(n). 
\end{equation}
Analogous to the Euclidean Hessian, the Riemannian Hessian characterizes how the Riemannian gradient varies as $Q$ varies along a tangent direction $E \in T_Q O(n)$. For the Riemannian embedded submanifold $O(n)$, the Riemannian Hessian can be computed by first computing the directional derivative of $\textrm{grad}f$ in the direction $E$ and then projecting the result onto the tangent space $T_Q O(n)$ using \eqref{eq:orthoproj} \cite{boumal2023introduction}. Here, $\textrm{grad}f$ is a map which takes $Q \in O(n)$ to the Riemannian gradient \eqref{eq:riemannian_grad}. 

Taking the differential of $\textrm{grad}f$ at $Q$ defined by \eqref{eq:riemannian_grad} in the direction $E \in T_Q O(n)$ and using linearity of the differential we obtain 
\begin{equation} \label{eq:differential_of_Rgrad}
D\textrm{grad}f(Q)[E] = D \textrm{Proj}_Q(\nabla_Q f(Q)) [E] = \frac{ D \nabla_Q f(Q)[E] - D\left(Q\nabla_Qf(Q)^{\top}Q\right)[E]}{2}. 
\end{equation}
This first term in the numerator is the differential of the Euclidean gradient $\nabla_Qf(Q)$. Applying the product rule to the second term in the numerator of \eqref{eq:differential_of_Rgrad} we have
\begin{align} 
D\left(Q \nabla_Q f(Q)^{\top} Q\right)[E] &= E \nabla_Q f(Q)^{\top} Q + Q \left(D\nabla_Qf(Q)[E]^{\top}\right) Q + Q \nabla_Q f(Q)^{\top} E, 
\end{align}
which also requires the differential of the Euclidean gradient $\nabla_Qf(Q)$. The differential of the Euclidean gradient $\nabla_Qf(Q)$ (given in \eqref{eq:euclidean_grad}) can be expressed as 
\begin{equation} \label{eq:diff_prod_rule}
\begin{aligned}
D \nabla_Q f(Q) [E] &= D \left({\bf A}^{-1}\left(U \phi'(\Sigma) V^{\top}\right)X^{\top}\right)[E] \\
&= {\bf A}^{-1}\left(D U[E] \phi'(\Sigma) V^{\top} +  U D\phi'(\Sigma)[E]V^{\top} + U\phi'(\Sigma)DV^{\top}[E]\right)X^{\top},
\end{aligned}
\end{equation}
where we used the fact that ${\bf A}^{-1}$ is linear and the product rule. The differentials of $U$, $\phi'(\Sigma)$ and $V$ in \eqref{eq:diff_prod_rule} we obtain using known results for differentiating the singular value decomposition \cite{townsend}. Using the fact that $D\bfA{QX}[E]=\bfA{EX}$ we have 
\begin{equation} \label{eq:svd_diff}
\begin{aligned}
DU[E] &= U\left[F \odot \left(U^{\top} \bfA{EX} V\Sigma + \Sigma V^{\top} \bfA{EX}^{\top} U\right)\right] + \left(I_{lc} - UU^{\top}\right)\bfA{EX} V \Sigma^{-1} \\ 
D\phi'(\Sigma)[E] &= \phi''(\Sigma)  D\Sigma [E] = \phi''(\Sigma) U^{\top} \bfA{EX}V \\ 
DV[E] &= V\left[F \odot \left(\Sigma U^{\top} \bfA{EX}V + V^{\top}\bfA{EX}^{\top} U\Sigma\right) \right] + \left(I_{rb} - VV^{\top}\right) \bfA{EX}^{\top} U\Sigma^{-1}, \\ 
\end{aligned}
\end{equation} 
where 
\begin{equation} \label{eq:F_def}
F_{ij} = \begin{cases}
         \displaystyle\frac{1}{\sigma_j^2 - \sigma_i^2}, \qquad i\neq j\\
         0, \qquad i = j, 
         \end{cases}
\end{equation}
$\odot$ denotes the Hadamard product and $I_m$ denotes the $m \times m$ identity matrix. 
Equations \eqref{eq:differential_of_Rgrad}-\eqref{eq:F_def} provide a closed form expression for the differential $D \nabla_Q f(Q)[E]$ of the Euclidean gradient $\nabla_Q f(Q)$. The Riemannian Hessian of $f$ at $Q$ in the direction $E$ is now easily obtained by projecting the differential $D \textrm{grad} f(Q)[E]$ orthogonally onto the tangent space 
\begin{equation} \label{eq:hessian}
\textrm{Hess} f(Q)[E] = \textrm{Proj}_Q \left(D\textrm{grad}f(Q)[E]\right). 
\end{equation}
We summarize the main steps for computing the Riemannian Hessian in the following algorithm. 
\begin{algorithm} \label{alg:Riemannian_hess}
 \caption{Compute Riemannian Hessian of disentangle cost function} 
\begin{algorithmic}[1]
\Require 
\Statex $X \in \mathbb{R}^{lr \times bc}$ $\rightarrow$ flattening of tensor ${\cal X}$ to be disentangled
\Statex $Q \in O(lr)$ $\rightarrow$ disentangler at current iteration
\Statex $E \in T_Q O(lr)$ $\rightarrow$ tangent direction along which the Hessian is evaluated
\Ensure 
$\textrm{Hess} f(Q)[E]$ $\rightarrow$ Riemannian Hessian at $Q$ in the direction $E$
\State $Y = \bfA{QX}$
\State $[U,\Sigma,V] = \texttt{svd}(Y)$ 
\State $\nabla_Q f(Q) = {\bf A}^{-1}(U\phi'(\Sigma)V^{\top})X^{\top}$ \Comment{Euclidean gradient} 
\State $\textrm{grad}f(Q)=\textrm{Proj}_Q\left(\nabla_Q f(Q)\right)$ \Comment{Riemannian gradient}
\State Evaluate $F$ in \eqref{eq:F_def}
\State Evaluate $DU[E], D\phi'(\Sigma)[E], DV[E]$ in \eqref{eq:svd_diff} 
\State Evaluate $D\nabla_Qf(Q)[E]$ using \eqref{eq:diff_prod_rule}
\State Evaluate $D \textrm{grad}f(Q)[E]$ using \eqref{eq:differential_of_Rgrad} 
\State $\textrm{Hess} f(Q)[E] = \textrm{Proj}_Q \left(D\textrm{grad}f(Q)[E]\right)$ 
\end{algorithmic}
\end{algorithm}

\noindent
We estimate the computational complexity of Algorithm~\ref{alg:Riemannian_hess} assuming $lc \leq rb$ for simplicity. 
\begin{enumerate}[label={Step \arabic*:}, leftmargin=*]
\item[Steps 1--4:] Compute the Riemannian gradient costing $\mathcal{O}((lr)^3 + (lr)^2bc + (lc)^2 rb)$, as discussed in Section~\ref{sec:riemannian_grad}.

\item[Step 5:] The cost of the $lc(lc-1)$ non-zero elements of $F\in \mathbb{R}^{lc \times lc}$ is $\mathcal{O}((lc)^2)$ which is negligible. 

\item[Step 6:] First compute ${\bf A}(EX)$, where $E \in \mathbb{R}^{lr \times lr}$ and $X \in \mathbb{R}^{lr \times bc}$, which costs $\mathcal{O}((lr)^2 bc)$. Then compute $U^{\top}{\bf A}(EX)V$, where $U \in \mathbb{R}^{lc \times lc}$, ${\bf A}(EX) \in \mathbb{R}^{lc \times rb}$, and $V \in \mathbb{R}^{rb \times lc}$, which costs $\mathcal{O}((lc)^2rb)$, and $V^{\top}{\bf A}(EX)^{\top} U$, which costs  $\mathcal{O}((lc)^3)$. Diagonal scaling by $\Sigma, \phi''(\Sigma)\in \mathbb{R}^{lc \times lc}$ and Hadamard product with $F$ have a negligible cost of $\mathcal{O}((lc)^2)$. The total cost for $D\phi'(\Sigma)[E]$ and the first term in $DU[E], DV[E]$ is $\mathcal{O}((lc)^3 + (lr)^2bc + (lc)^2rb)$. Finally $\left(I_{lc} - UU^{\top}\right)\bfA{EX} V \Sigma^{-1}$ costs $\mathcal{O}((lc)^3 + (lc)^2rb)$ and $\left(I_{rb} - VV^{\top}\right) \bfA{EX}^{\top} U\Sigma^{-1}$ costs $\mathcal{O}((lr)^2bc + (lc)^2rb)$. The total cost in Step 6 is $\mathcal{O}((lc)^3 + (lr)^2bc + (lc)^2rb)$. 

\item[Step 7:] The dominant costs are from multiplication involving matrices $U, DU[E], D\phi'(\Sigma)[E] \in \mathbb{R}^{lc \times lc}$ and $V, DV[E] \in \mathbb{R}^{lc \times rb}$, costing $\mathcal{O}((lc)^2rb)$, and multiplication of ${\bf A}^{-1}(\bullet) \in \mathbb{R}^{lr \times bc}$ with $X^{\top} \in \mathbb{R}^{bc \times lr}$, costing $\mathcal{O}((lr)^2bc)$. Diagonal scaling with $\phi'(\Sigma)$ has negligible cost. The total cost for Step 7 is thus $\mathcal{O}((lc)^2rb + (lr)^2bc)$.

\item[Step 8:] The difference of two matrices with size $lr \times lr$ and rescaling by $1/2$ has cost $\mathcal{O}((lr)^2)$, which is negligible. 

\item[Step 9:] The dominant cost in the orthogonal tangent space projection is matrix multiplication between matrices of size $lr \times lr$, which scales as $\mathcal{O}((lr)^3)$. 

\end{enumerate}
Adding the costs from each step, the total computational complexity for the Riemannian Hessian is $\mathcal{O}((lr)^3 + (lc)^3 + (lr)^2bc + (lc)^2 rb)$.

\subsection{Riemannian nonlinear conjugate gradient and trust-region Newton methods}
The Riemannian gradient defined in \eqref{eq:riemannian_grad} can be used in a Riemannian nonlinear conjugate gradient (RCG) algorithm to obtain an optimal disentangler $Q^*$.  In an RCG algorithm, the approximate $Q$ matrix is updated as
\begin{equation}
Q^{(j+1)} = \mathcal{R}(Q^{(j)},t^{(j)} \Gamma^{(j)}),
\end{equation}
where $\Gamma^{(j)}$ is a Riemannian search direction defined at the $j$-th approximation $Q^{(j)}$, $t^{(j)}$ is an appropriate step length, and $\mathcal{R}$ denotes a retraction operation that brings $Q^{(j)} + t^{(j)}\Gamma^{(j)}$, which is the update of $Q^{(j)}$ on the tangent space, back to the Riemannian manifold $O(n)$.

In RCG, a conjugate search direction $\Gamma$ is updated as
\begin{equation}
\Gamma^{(j)} = -{\rm grad}f(Q^{(j)}) + \beta^{(j)} \mathcal{T}(\Gamma^{(j-1)}),
\end{equation}
where $\mathcal{T}$ denotes a parallel transport operator that maps a search direction in $T_{Q^{(j-1)}}O(n)$ to a search direction in $T_{Q^{(j)}}O(n)$, and the parameter $\beta^{(j)}$ is chosen to ensure that $\Gamma^{(j)}$ is approximately ``conjugate" to the previous search direction with respect to the Hessian. In RCG, $t^{(j)}$ can be updated via an adaptive line search along the manifold~\cite{Sato2022}, while $\beta^{(j)}$ can be determined by the Riemannian generalizations of the Fletcher-Reeves~\cite{fletcher_reeves},  Polak-Ribi\`ere-Polyak~\cite{polak_ribiere,polyakConjugateGradientMethod1969}, Hestenes-Stiefel~\cite{hestenesMethodsConjugateGradients}, and other methods~\cite{daiyuan99,liustorey91}. The basic steps of an RCG algorithm is given in Algorithm~\ref{alg:rcg}. We refer readers to~\cite{Sato2022} for details on how $t^{(j)}$ and $\beta^{(j)}$ are computed, and how the parallel transport $\mathcal{T}$ and the retraction $\mathcal{R}$ are performed. 
\begin{algorithm} \label{alg:rcg}
 \caption{A Riemannian conjugate gradient (RCG) algorithm.} 
\begin{algorithmic}[1]
\Require 
\Statex $Q^{(0)} \in O(n) \rightarrow$ initial disentangler
\Statex $\texttt{maxiter} \rightarrow$ maximum number of iterations allowed 
\Statex $\epsilon \rightarrow$ convergence tolerance 
\Ensure $Q^*$ is a local minimizer of $f(Q)$ in $O(n)$.

\State $j=0$
\State $\Gamma^{(0)} = {\rm grad} f(Q^{(0)})$
\State converged = false
\While{(Not converged and $j < \texttt{maxiter}$)} 
   \State Compute a step length $t^{(j)} > 0$
   \State $Q^{(j+1)} \leftarrow \mathcal{R}(Q^{(j)},t^{(j)}\Gamma^{(j)})$
   \State Compute $G^{(j+1)} = {\rm grad}f(Q^{(j+1)})$
   \If {($\|G^{(j+1)}\| \leq \epsilon$)}
      \State converged = true
      \State $Q^* = Q^{(j+1)}$
   \Else
      \State Compute $\beta^{(j+1)}$
      \State $\Gamma^{(j+1)} \leftarrow -G^{(j+1)} + \beta^{(j+1)} \mathcal{T}^{(j)}(\Gamma^{(j)})$ 
      where $\mathcal{T}^{(j)}:T_{Q^{(j)}}\rightarrow T_{Q^{(j+1)}}$ is a parallel transport operator
   \EndIf
\State $j \leftarrow j + 1$
\EndWhile
\end{algorithmic}
\end{algorithm}

With the Riemannian Hessian available, we can use second-order Riemannian optimization methods. In step $j$ of these methods, one can move along a Newton search direction $\Gamma^{(j)}$ that satisfies the Newton correction equation 
\begin{equation}
\textrm{Hess} f(Q^{(j)})[\Gamma^{(j)}] = -\textrm{grad} f(Q^{(j)}).
\label{eq:newtoneq}
\end{equation}
This is followed by a retraction step that places the new approximation back onto the manifold $O(n)$~\cite{boumal2023introduction}.

Note that the search direction $\Gamma^{(j)}$ is within the tangent space associated with $O(n)$ at the $j$-th approximation $Q^{(j)}$. If $\textrm{Hess} f(Q^{(j)})$ is positive definite, $\Gamma^{(j)}$ is guaranteed to be a descent direction. Taking the full Newton step yields the minimizer of the local quadratic approximation of $f(Q)$ over all $E \in T_{Q^{(j)}} O(n)$, which is defined as
\begin{equation}
m_{Q^{(j)}}(\mathcal{R}(E)) = f(Q^{(j)}) + \langle \textrm{grad}f(Q^{(j)}), E \rangle + \frac{1}{2}\langle \textrm{Hess} f(Q^{(j)})[E], E \rangle. 
\label{eq:quadmodel}
\end{equation}

However, if the Riemannian Hessian is not positive definite but non-singular, the convergence of a naive implementation of the Newton's method can be problematic unless the initial guess of the minimizer is sufficiently close to the solution.

One way to address convergence failure is to modify the Hessian by replacing it with a positive-definite operator. One widely used technique is to simply replace $\textrm{Hess} f(Q^{(j)})[E]$ on the left-hand side of \eqref{eq:newtoneq} with 
\begin{equation}
\textrm{Hess} f(Q^{(j)})[E] + \eta E,
\label{eq:regHess}
\end{equation}
where $\eta$ is an appropriately chosen \textit{regularization} parameter to make the operator on the left-hand side of \eqref{eq:newtoneq} positive definite~\cite{CGTbook00}.

Because $m_{Q^{(j)}}(\mathcal{R}(E))$ is a local approximation of $f(Q)$ on the tangent space defined at $Q^{(j)}$, taking a full Newton step may not be appropriate when the minimizer of $m$ is far from $Q^{(j)}$. One way to address this issue is to impose an additional constraint that defines a small region around $Q^{(j)}$ in which $m_{Q^{(j)}}(\mathcal{R}(E))$ can be trusted to be a good approximation of $f(Q)$. The search direction is obtained as the solution of the \textit{trust-region} subproblem
\begin{equation}
\min_{\|E\| \leq \Delta} m_{Q^{(j)}}(\mathcal{R}(E)),
\label{eq:trsubprob}
\end{equation}
where $\Delta$ is a trust-region radius that is adjusted dynamically based on the ratio of the actual reduction in $f(Q)$ over the reduction in $m_{Q^{(j)}}$ when a optimal solution of \eqref{eq:trsubprob} is used to update $Q^{(j)}$.

Because $\textrm{Hess} f(Q^{(j)})$ is only available through its multiplication with a search direction $E$ as defined in \eqref{eq:differential_of_Rgrad}--\eqref{eq:hessian}, the Newton correction equation or the trust-region subproblem is often solved iteratively using a modified conjugate gradient method~\cite{steihaug_cg}. We refer readers to~\cite{boumal2023introduction} on the algorithmic details of how these problems are solved in a trust-region Newton's method. The major steps of a Riemannian trust-region Newton's (RTRN) method are given in Algorithm~\ref{alg:rtr-Newton}.
\begin{algorithm}[htbp] \label{alg:rtr-Newton}
 \caption{A Riemannian trust-region Newton (RTRN) algorithm.} 
\begin{algorithmic}[1]
\Require 
\Statex $Q^{(0)} \in O(n) \rightarrow$ initial disentangler
\Statex $\Delta_{\max} \rightarrow$ maximum trust-region radius 
\Statex $\Delta_0 \in (0,\Delta_{\max}) \rightarrow$ initial trust-region radius 
\Statex $\texttt{maxiter} \rightarrow$ maximum number of iterations allowed 
\Statex $r_{\min} \rightarrow$ minimum reduction ratio 
\Statex $\epsilon \rightarrow$ convergence tolerance 
\Ensure $Q^*$ is a local minimizer of $f(Q)$ in $O(n)$.

\State $j=0$
\State $\Delta = \Delta_0$
\State converged = false
\While{(Not converged and $j < \texttt{maxiter}$)} 
   \State Solve \eqref{eq:trsubprob} and set the tentative next approximation to $\tilde{Q}^{(j+1)} = \mathcal{R}(E_{\textrm{opt}})$
   \State Compute ratio $r = \frac{f(\tilde{Q}^{(j+1)})-f(Q^{(j)})}{m(\tilde{Q}^{(j+1)})-m(0)}$
   \State Accept or reject $\tilde{Q}^{(j+1)}$:
   \[
   Q^{(j+1)} = \left\{ \begin{array}{ll}
   \tilde{Q}^{(j+1)} & \mbox{if $r > r_{\min}$ (accept),} \\
   Q^{(j)} & \mbox{otherwise (reject)}
   \end{array}
   \right.
   \]
   \State Update the trust region radius:
   \[
   \Delta \leftarrow \left\{ \begin{array}{ll}
   \Delta/4 & \mbox{if $r < \frac{1}{4}$}  \\
   \min(2\Delta,\Delta_{\max}) & \mbox{if $r > \frac{3}{4}$ and $\|E_{\textrm{opt}}\| = \Delta$} \\
   \end{array}
   \right.
   \]
   \If {($\|\textrm{grad}f(Q^{(j+1)})\| \leq \epsilon$)}
      \State converged = true
      \State $Q^* = Q^{(j+1)}$
   \EndIf
\State $j \leftarrow j + 1$
\EndWhile
\end{algorithmic}
\end{algorithm}

\section{Alternating optimization}
\label{sec:altopt}
As an alternative to the gradient-based optimization algorithms for computing the disentangler $Q$ described in the preceding sections, in this section we introduce an alternating least-squares scheme. We consider $\phi$ given in \eqref{eq:Heaviside} and reformulate the optimization problem \eqref{eq:opt_constrained} as 
\begin{equation} \label{eq:opt_on_product_mfld}
\min_{(Q,M_k) \in O(lr) \times \mathbb{R}_k^{lc \times rb}} 
\|  {\bf A} (QX) - M_k \|_F, 
\end{equation}
where $\mathbb{R}_k^{lc \times rb}$ is the manifold of rank-$k$ real matrices with size $lc \times rb$ and $O(lr) \times \mathbb{R}_k^{lc \times rb}$ is a product manifold. 
To see that \eqref{eq:opt_on_product_mfld} is equivalent to \eqref{eq:opt_constrained}, notice that for any $Q$ the optimal rank-$k$ matrix $M_k$ is given by the rank-$k$ truncated SVD of $\bfA{QX}$.  Fixing $M_k$ in \eqref{eq:opt_on_product_mfld} as the optimal truncated SVD yields an optimization problem for the disentangler $Q$ 
\begin{equation}
\label{eq:opt_SVs}
\begin{aligned}
\min_{Q \in O(lr)}\sqrt{ \sum_{i = k+1}^{m} \sigma_i^2(\bfA{QX})}, 
\end{aligned}
\end{equation}
which is equivalent to \eqref{eq:opt_constrained} with $\phi(t)=\phi_{\mu}(t)$ defined in \eqref{eq:Heaviside}. The advantage of considering \eqref{eq:opt_on_product_mfld} is that for fixed $Q$ we have a least-squares problem for $M_k$, and for fixed $M_k$ we have a least-squares problem for $Q$. Both of these least-squares problems have a simple, closed-form solution that is relatively cheap to compute. Alternating between these least-squares problems yields a sequence of disentanglers $Q^{(j)}$ for $j=1,2,\ldots$ that converges to a (local) minimum of \eqref{eq:opt_SVs}.

To describe the alternating algorithm, we begin with a initial disentangler $Q^{(0)}$, e.g., the identity matrix $Q^{(0)} = I_{lr \times lr}$, to use as our zeroth iteration. We obtain iteration $Q^{(j+1)}$ from $Q^{(j)}$ as follows. Fix $Q^{(j)}$ in \eqref{eq:opt_on_product_mfld} and obtain $M_k^{(j)}$ by solving 
\begin{equation}
    M_k^{(j)} = \argmin_{M_k \in \mathbb{R}_k^{lc \times rb}} 
    \left\| {\bf A}\left(Q^{(j)}X\right) - M_k\right\|_F. 
\end{equation}
The preceding optimization problem is solved directly by taking the rank-$k$ truncated SVD $M_k^{(j)} = \texttt{SVD}_k\left({\bf A}\left(Q^{(j)}X\right)\right)$. Then with $M_k^{(j)}$ fixed in \eqref{eq:opt_on_product_mfld} we obtain the updated disentangler $Q^{(j+1)}$ by solving 
\begin{equation}
\label{eq:procrustes_orth0}
    Q^{(j+1)} = \argmin_{Q \in O(lr)} \left\|{\bf A}\left(QX\right) - M_k^{(j)} \right\|_F. 
\end{equation}
To obtain a closed form solution to \eqref{eq:procrustes_orth0}, it is convenient to consider the equivalent problem 
\begin{equation}
\label{eq:procrustes_orth}
     Q^{(j+1)} = \argmin_{Q \in O(lr)} 
     \left\| QX - {\bf A}^{-1} \left(M_k^{(j)}\right)\right\|_F, 
\end{equation}
known as an orthogonal Procrustes problem \cite{Procrustes}, which is solved by the polar decomposition 
\begin{equation}
     Q^{(j+1)} = UV^{\top}. 
\end{equation}
Here $U,V$ are the left, right factors in the SVD of the matrix ${\bf A}^{-1}\left(M_k^{(j)}\right)X^{\top}$. Iterating this procedure generates a sequence of disentanglers $Q^{(j)}$ and rank-$k$ matrices $M_k^{(j)}$ that reduce the objective function \eqref{eq:opt_on_product_mfld} monotonically and hence converge to a (local) minimum of the optimization problem \eqref{eq:opt_on_product_mfld}. 
We remark that the Procrustes problem also appears in a disentangling algorithm that minimizes the Rényi-2 entropy~\cite{hauschildFindingPurificationsMinimal2018}. More broadly, the alternating disentangler is a particular instance of Evenbly-Vidal optimization~\cite{evenblyAlgorithmsEntanglementRenormalization2009}. 

The main steps of the alternating disentangler algorithm are summarized below. To determine when the alternating disentangler algorithm has converged we consider 
\begin{equation} \label{eq:dQ}
    \delta Q^{(j+1)} = \left\| Q^{(j+1)}-Q^{(j)} \right\|_F,
\end{equation}
which measures how much the disentangler has changed from iteration $j$ to $j+1$. Once \eqref{eq:dQ} falls below a user-determined threshold, we say that the alternating algorithm has converged.

\begin{algorithm} \label{alg:alternating}
 \caption{Alternating disentangler algorithm} 
\begin{algorithmic}[1]
\Require 
\Statex $X \in \mathbb{R}^{lr \times bc}$ $\rightarrow$ flattening of tensor ${\cal X}$ to be disentangled
\Statex $Q^{(0)} \in O(lr) \rightarrow$ initial disentangler
\Statex $\texttt{maxiter} \rightarrow$ maximum number of iterations allowed 
\Statex $\epsilon \rightarrow$ convergence tolerance 
\Ensure 
$Q^{*}\in O(lr)$ $\rightarrow$ local solution to \eqref{eq:opt_SVs} 
\State $j=0$
\State converged = false
\While{(Not converged and $j < \texttt{maxiter}$)}
\State $M^{(j)}_k = \texttt{SVD}_k\left(\bfA{Q^{(j)}X}\right)$ 
\State $[U,\Sigma,V] = \texttt{SVD}\left({\bf A}^{-1}\left(M^{(j)}_k\right) X^{\top}\right)$
\State $Q^{(j+1)} = UV^{\top}$
\If {($\left\| Q^{(j+1)}-Q^{(j)} \right\|_F \leq \epsilon$)}
  \State converged = true
  \State $Q^* = Q^{(j+1)}$
\EndIf
\State $j \leftarrow j + 1$
\EndWhile 
\end{algorithmic}
\end{algorithm}

We estimate the computational complexity of a single iteration (Steps 3--5) of the alternating disentangler in Algorithm~\ref{alg:alternating}. 
\begin{enumerate}[label={Step \arabic*:}, start=3, leftmargin=*]
\item Compute the matrix product between $Q^{(j)} \in \mathbb{R}^{lr \times lr}$ and $X \in \mathbb{R}^{lr \times bc}$ costing $\mathcal{O}((lr)^2bc)$. Then compute the rank-$k$ truncated SVD on ${\bf A}(Q^{(j)}X)$ costing $\mathcal{O}(k lc rb)$. 

\item Compute the product of ${\bf A}^{-1}(M_k^{(j)}) \in \mathbb{R}^{lr \times bc}$ and $X^{\top} \in \mathbb{R}^{bc \times lr}$ costing $\mathcal{O}((lr)^2 bc)$. Then compute the SVD of ${\bf A}^{-1}(M_k^{(j)})X^{\top} \in \mathbb{R}^{lr \times lr}$ costing $\mathcal{O}((lr)^3)$. 

\item The disentangler is computed as the matrix product $UV^{\top}$ where $U,V \in \mathbb{R}^{lr \times lr}$ costing $\mathcal{O}((lr)^3)$. 
\end{enumerate}
Adding the costs from each step, the total computational complexity of a single iteration of the alternating disentangler is
$\mathcal{O}((lr)^3 + (lr)^2 bc)$. 

\section{Practical Issues}
\label{sec:practical}
We discuss a few practical issues to be considered in the implementation of the disentangling optimization algorithms. 

\subsection{Rank initialization and determination}
\label{sec:rank}
Minimizing the objective function \eqref{eq:sv_tail} by Riemannian or alternating optimization algorithms requires us to provide a rank parameter $k$ in \eqref{eq:opt_on_product_mfld}. Ideally, we would like $k$ to be as small as possible while keeping the optimized objective function in \eqref{eq:opt_on_product_mfld} sufficiently small. However, such a minimal rank is generally unknown a priori. We will show that it is possible to obtain an estimate of the minimal rank by either counting the
number of degrees of freedom in the disentangled tensor representation of $X$.

Suppose the disentangler $Q$ allows us to achieve 
\begin{equation}
\sum_{i=k+1}^m \sigma_i^2(\mathbf{A}(QX)) = 0,
\label{eq:exactrank}
\end{equation}
for some $k$. This implies we can find a $Q$ parameterized by $e^{B}$, where $B$ is a skew-symmetric matrix such that $M_k = \mathbf{A}(QX)$ is an $lc\times rb$ rank-$k$ matrix.  We can parameterize $M_k$ as the product of $Y$ and $W^{\top}$, where $Y \in \mathbb{R}^{lc\times k}$ and $W \in \mathbb{R}^{rb \times k}$.  Clearly, $Y$, $W$, and $Q$ satisfy the equation
\begin{equation}
Q^{\top}\mathbf{A}^{-1}(YW^{\top}) = X. \label{eq:qywx}
\end{equation}
If the matrix elements in the lower- (or upper-) triangular part of $B$ as well as those in $W$ and $Y$ are treated as unknowns, the solution of \eqref{eq:qywx} may exist if the number of unknowns matches the number of (scalar) equations, which is $lrbc$.

A naive sum of the degrees of freedom on the left-hand side of \eqref{eq:qywx} includes $klc$ and $krb$ degrees of freedom in $Y$ and $W$, respectively, and $lr(lr-1)/2$ degrees of freedom in $B$ (recall $Q = e^{B}$), yielding $k(lc + rb) + \frac{lr(lr-1)}{2}$.   

However, we should note that orthogonal matrices of the form 
\[
Q = Q_l \otimes Q_r, 
\]
where $Q_l$ and $Q_r$ can be represented as 
$e^{B_l}$ and $e^{B_r}$ for some $l\times l$ and $r \times r$ skew-symmetric matrices $B_l$ and $B_r$ respectively, do not contribute to the disentangling of $\mathcal{X}$. Therefore, we should exclude the number of degrees of freedom in $Q_l$ and $Q_r$, which are $l(l-1)/2$ and $r(r-l)/2$ respectively. Furthermore, the choices of $W$ and $Y$ are not unique because $YW^{\top} = YSS^{-1}W^{\top} = \hat{Y} \hat{W}^{\top}$, where $\hat{Y}=YS$ and $\hat{W} = WS^{-{\top}}$, for any non-singular matrix $S$. Hence, we need to remove $k^2$ degrees of freedom from the gauge matrix $S$ also. (For completeness, if $lr > bc$, we have to subtract the additional gauge degrees of freedom corresponding to an orthogonal matrix of dimension $(lr - bc) \times (lr - bc)$; below, we consider $lr \leq bc$ only, which is the case for most applications.)

As a result, when $lr \leq bc$, the total number of degrees of freedom on the left-hand side of \eqref{eq:qywx} is 

\begin{equation}
k(lc + rb) - k^2 + \frac{lr(lr-1)}{2} - \frac{l(l-1)}{2} - \frac{r(r-1)}{2}. 
\label{eq:exactrankdim}
\end{equation}
This number must be at least $lrbc$, which is the number of degrees of freedom in $X$ (and the number of scalar equations in \eqref{eq:qywx}), to ensure \eqref{eq:qywx} can be solved in general.

As an example, let us first consider 
an equal-dimensional $l\times l\times l\times l$ tensor.
To ensure \eqref{eq:qywx} is solvable, we should choose a minimal $k$ to satisfy
\begin{equation}
2l^2k - k^2 + \frac{l^2(l^2-1)}{2} - l(l-1) \geq l^4.
\end{equation}
For $ l \geq 2$, it follows that the optimal rank $k_l^*$ can be evaluated as
\begin{equation}
k^*_l = \lceil l^2 - \frac{\sqrt{2}}{2}\sqrt{(l-1)^2l(l+2)}\rceil
\label{eq:exactrankn}
\end{equation}

For large $l$, $k^*_l \approx \frac{2 - \sqrt{2}}{2} l^2 + O(l) \approx 0.29l^2$. 

Second, we consider the case in which $l=r$ and $b=c$. For $c\geq l \geq 2$, the optimal rank $k^*_{l,c}$ is 

\begin{equation}
k^*_{l,c} = \lceil lc - \frac{\sqrt{2}}{2}\sqrt{(l-1)^2l(l+2)}\rceil
\label{eq:exactrankmn}
\end{equation}

In the limit where $l \ll c$, we observe that $k^*_{l,c}$ is nearly equal to  $lc$, up to a rank reduction of $\sim \frac{\sqrt{2}}{2} l^2$. This matches our expectation that a small disentangler is ineffective if further assumptions are not made about $X$. 

It should be noted that the rank estimate determined by~\eqref{eq:exactrankdim} does not take into account the numerical value of each element in $X$ or the accuracy requirement of the low-rank approximation. Therefore, the estimated rank can be larger than the optimal one.
Nonetheless, the estimate given by Equation \ref{eq:exactrankdim} can serve as a starting point for choosing the rank when benchmarking different optimization methods. For example, by choosing $k \geq k^*$, we can test how well different optimizers find an approximate solution to Equation \ref{eq:exactrank} by minimizing Equation \ref{eq:sv_tail}. 

\subsection{Rank minimization}
We say that $Q$ is a good rank-$k$ disentangler if 
\begin{equation}
c_k(Q) \le \epsilon,
\label{eq:convcheck}
\end{equation}
for some user-provided truncation error tolerance $\epsilon$. However, as we mentioned earlier, we would also like $k$ to be as small as possible. Therefore, in addition to solving the optimization problem \eqref{eq:opt_constrained} or \eqref{eq:opt_on_product_mfld}, we would like to seek an optimal $k$. We propose to identify such a minimal $k$ by using a binary search algorithm given in Figure~\ref{alg:bisection}.
In this algorithm, we initialize left and right bounds for searching for the optimal $k$ at $k_l$ and $k_r$. They can be set to $0$ and $\min(lc, rb)$ if no better bounds are known; in practice, we optimize a surrogate function such as the R\'{e}nyi-1/2 entropy (\ref{eqn:renyi}) and evaluate $c_k$ at the resulting disentangler to initialize $k_r$. 

The next estimate of the optimal rank is taken to be $k=\lfloor(k_l+k_r)/2\rfloor$. For each $k$, we use either a Riemannian optimization or the alternating minimization algorithm to find an optimal disentangler $Q_k$. We then examine the value of $c_k(Q_k)$. If $c_k(Q_k) \leq \epsilon$, we set $k_r$ to $k-1$. Otherwise, we set $k_l$ to $k+1$. While this ensures that the algorithm will terminate when $k_l > k_r$, we note that $k_r$ is not a strict upper bound. As a result, we update the optimal upper bound variable $k_{\rm opt}=k$ whenever (\ref{eq:convcheck}) is satisfied, which is returned when the algorithm terminates. 

\begin{algorithm}[htbp] \label{alg:bisection}
 \caption{Binary search rank selection and disentangler optimization} 
\begin{algorithmic}[1]
\Require 
\Statex $X \in \mathbb{R}^{lr \times bc}$ $\rightarrow$ flattening of tensor ${\cal X}$ to be disentangled
\Statex $Q^{(0)} \in O(lr)$ $\rightarrow$ initial disentangler
\Statex $\epsilon$ $\rightarrow$ truncation error tolerance 
\Statex $\texttt{maxiter}$ $\rightarrow$ maximum number of disentangler optimization routines
\Ensure
\Statex $k_{\rm opt}$ $\rightarrow$ truncation rank
\Statex $Q$ $\rightarrow$ disentangler such that truncation error is bounded by $\epsilon$, i.e., $c_k(Q) \leq \epsilon$
\State $k_l \leftarrow 0$
\State $Q$ $\leftarrow$ RCG disentangler optimization with Rényi-1/2 objective (\ref{eqn:renyi})
\State Determine $k_r$ such that $c_{k_r}(Q) \leq \epsilon$ 
\State $k_{\rm opt} \leftarrow k_r$
\State $\texttt{iter} \leftarrow 1$
\While{$k_l \leq k_r$ and $\texttt{iter} < \texttt{maxiter}$}
\State $k \leftarrow \lfloor(k_l +
k_r)/2 \rfloor$
\State $Q_k$ $\leftarrow$ Alternating or RCG optimization of rank-$k$ truncation error objective $c_{k}(Q)$. 
\If {$c_{k}(Q_k) \leq \epsilon$}
    \State $k_{\rm opt} \leftarrow k$
    \State $Q \leftarrow Q_k$
    \State $k_r \leftarrow k - 1$
\Else
    \State $k_l \leftarrow k + 1$
\EndIf
\State $\texttt{iter} \leftarrow \texttt{iter} + 1$
\EndWhile
\State return $k_{\rm opt}$, $Q$
\end{algorithmic}
\end{algorithm}

\subsection{Local minima}
We also note that disentangler optimization is a highly non-convex problem, irrespective of the chosen objective function. To illustrate the topography of the objective landscape, we consider a tensor $\chi \in \mathbb{R}^{4\times 4\times 4 \times 4}$ with uniformly distributed random entries in $[0,1]$. We normalize $X$ with respect to the Frobenius norm and consider the objective \eqref{eq:g_def} with $\phi$ chosen as the Heaviside function \eqref{eq:Heaviside} and $\mu$ corresponding to rank-$10$ truncation error, i.e.,  $\sigma_{10} < \mu < \sigma_{11}$. We set all but two entries of the vector $s$ equal to zero and in Figure~\ref{fig:cost_landscape} plot the objective as a function of two non-zero entries $s_1,s_2 \in [-2\pi,2\pi]$. The plot shows multiple local minima within this two-dimensional parameter space. While the local minima differ slightly in value, they are all sufficiently close to one another that, for practical purposes, any of them would yield a suitable disentangler. 

\begin{figure}[!hbtp]
    \centering
    \includegraphics[width=0.5\textwidth]{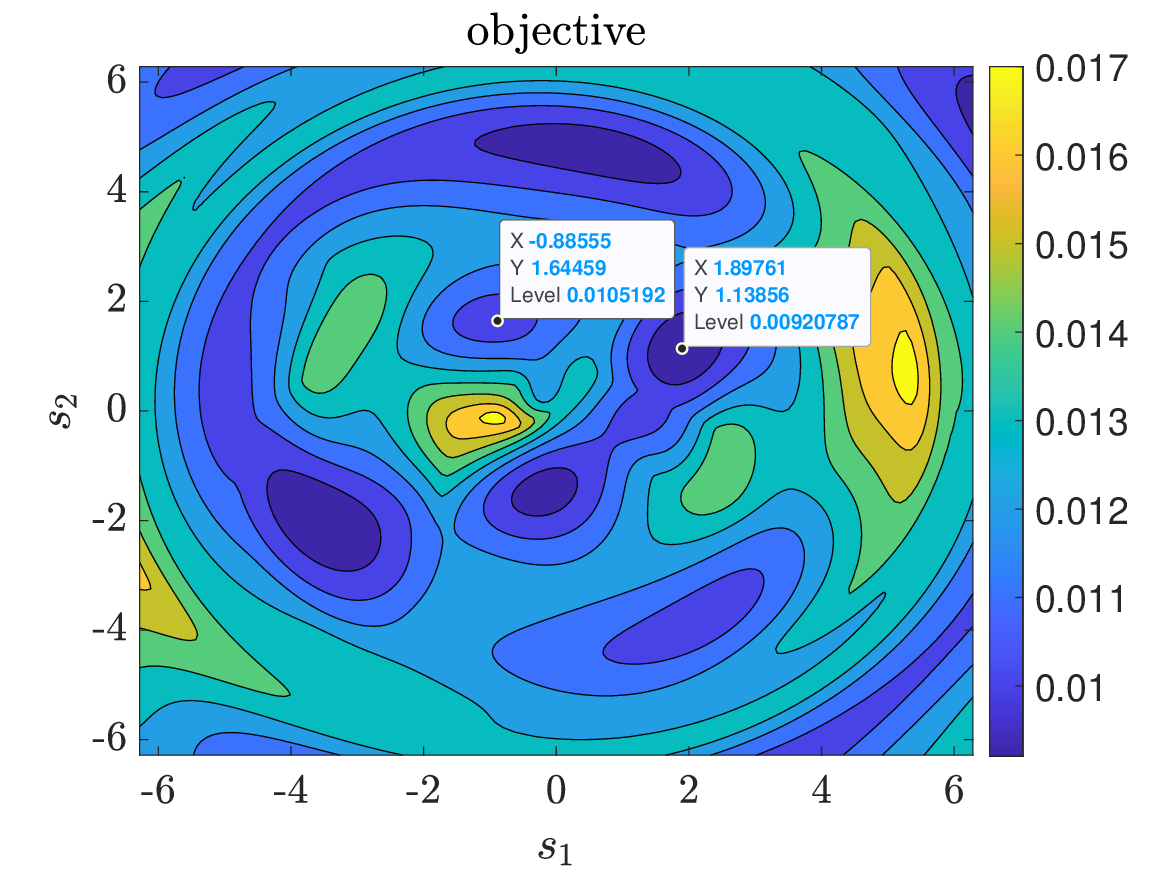}
    \caption{\label{fig:cost_landscape} Objective function \eqref{eq:g_def} measuring rank-$10$ truncation error for a $4 \times 4 \times 4 \times 4$ random tensor plotted over two skew-symmetric degrees of freedom $s_1,s_2 \in [-2\pi,2\pi]$. }
\end{figure}

\section{Numerical examples}
\label{sec:examples}
We test the optimization of disentanglers on two types of tensors. The first is a random tensor with entries drawn from a standard Gaussian distribution. The second is taken from an approximate isometric tensor network state (isoTNS) representation~\cite{zaletel2020isometric} of the ground state of the two-dimensional transverse-field Ising (2D TFI) Hamiltonian,
$
H = -J \sum_{\langle i, j \rangle} Z_i Z_j - g \sum_i X_i,
$
on an $L \times L$ lattice with open boundary conditions. IsoTNS is a tensor network ansatz with isometric constraints that enable efficient contraction and manipulation of 2D quantum states. An important procedure for isometric tensor networks is moving the orthogonality center, which can be done using the Moses Move~\cite{zaletel2020isometric}. The Moses Move consists of a sequence of isometric tensor ring decompositions~\cite{TensorRing2020}, which rewrites a region of the network in a form that shifts the orthogonality center. Disentangler optimization is a central step in the isometric tensor ring decomposition and directly impacts its accuracy and cost. Our second test tensor comes from an isoTNS approximation of the TFI ground state with $L=8$, $g=3.5$, $J=1$, computed using a 2D generalization of time-evolving block decimation  with imaginary-time step size $d\tau = 0.025$, bond dimension $\chi=12$, and orthogonality hypersurface bond dimension $\chi'=20$~\cite{zaletel2020isometric,lin2022efficient}. 

\begin{figure}[!hbtp]
    \centering
    \includegraphics[width=0.7\textwidth]{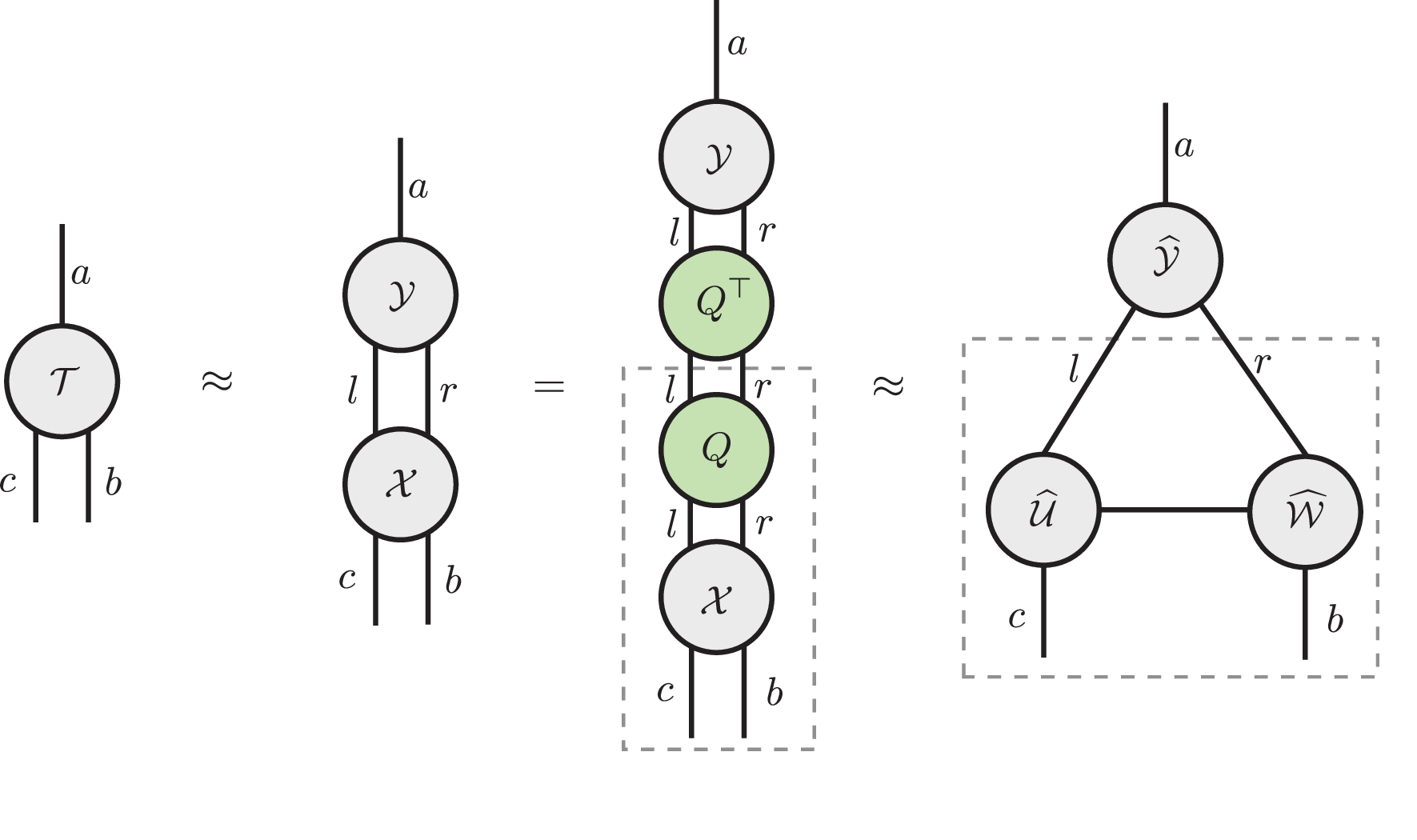}
    \caption{\label{fig:iso_tr} Isometric tensor ring decomposition. A three-dimensional tensor is decomposed using SVD with the resulting bond split into two dimensions indicated by $l$ and $r$. Then an optimized disentangler is inserted to minimize entanglement in $\mathcal{X}$. A second SVD completes the isometric tensor ring decomposition.}
\end{figure}

Hereafter we provide numerical comparisons of objective functions and optimization algorithms described in the preceding sections to our test tensors. All numerical experiments were performed using codes written in Python. We use the Pymanopt software to perform Riemannian optimization~\cite{JMLR:v17:16-177}. Timing measurements were collected using AMD ``Milan" CPUs, with computing resources from the National Energy Research Scientific Computing Center.
\subsection{Comparing different algorithms}
 We first compare different optimization methods for minimizing the objective function \eqref{eq:sv_tail} with a fixed $k$. The methods we compare include the alternating disentangler (Alt), Riemannian conjugate gradient (RCG), and Riemannian trust-region Newton (RTRN). For the RTRN algorithm, we try both the unregularized Hessian, and a regularized Hessian with the regularization parameter $\eta = 10^{-12}$ in \eqref{eq:regHess} (RTRN-reg). We use both the random and isoTNS tensors as test problems. For the equal-dimensional random tensor with size $n=12$, we set $k$ to 44 (following (\ref{eq:exactrankn})). For the isoTNS tensor, we set $k$ to 8. In all numerical experiments, the maximum number of iterations is set to 10,000 for Alt, 4,000 for RCG, and 1,000 for RTRN. The initial guess for the disentangler is set to the identity matrix $I_{lr}$. To determine convergence for Alt, the minimum change in the disentangler norm, $\delta Q$, is set to $10^{-12}$; in addition, we set the minimum change in the objective function norm to $10^{-12}$. To determine convergence for RCG and RTRN, the minimum norm of the Riemannian gradient is set to $10^{-8}$; the RCG method for updating the parameter $\beta^{(i)}$ is Hestenes-Stiefel.

In Figure~\ref{fig:methods_comp_iter}, we plot the change of the truncation error $c_k$ with respect to the iteration number for a few hundred iterations. We observe that, for the random tensor, RTRN has the fastest convergence rate, whereas the Alt exhibits the slowest convergence. The convergence rate of RCG is between the two, which is expected. The convergence of the regularized RTRN behaves similarly to that of RTRN, but truncation error is ultimately limited by the regularization parameter.

For the isoTNS tensor, the truncation error is quickly reduced  from $9 \times 10^{-4}$ to $8 \times 10^{-5}$ within a few hundred iterations of RTRN and Alt. Both algorithms appear to stagnate at this truncation error level from that point on, although Alt yields a slightly smaller truncation error. The convergence of RCG is slower and it stagnates at a higher truncation error. 
\begin{figure}[!hbtp]
\centering
\includegraphics[width=\textwidth]{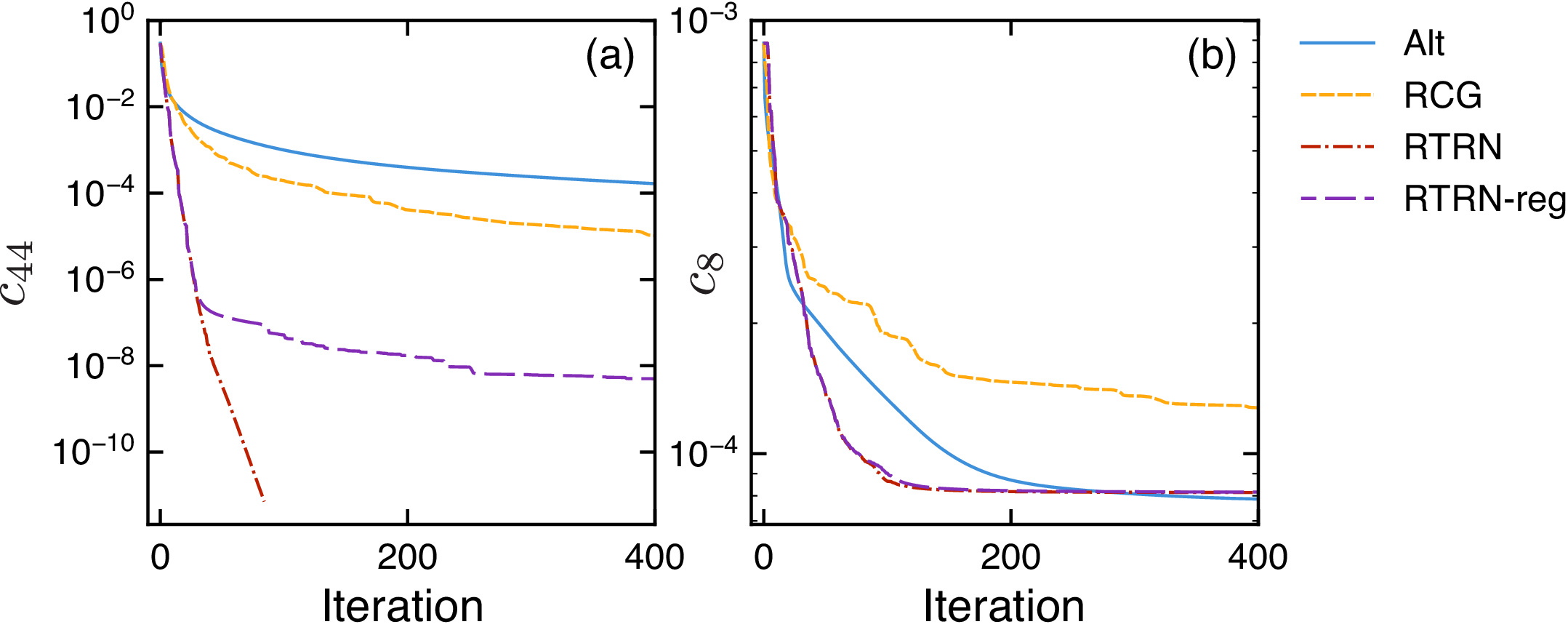}
\caption{\label{fig:methods_comp_iter} Iteration count of different methods for optimizing the truncation error at fixed rank $k$. (a) A random tensor of size (12, 12, 12, 12) with $k=44$. (b) An isoTNS tensor of size (8, 8, 60, 12) with $k=8$.}
\end{figure}
\begin{figure}[!hbtp]
\centering
\includegraphics[width=\textwidth]{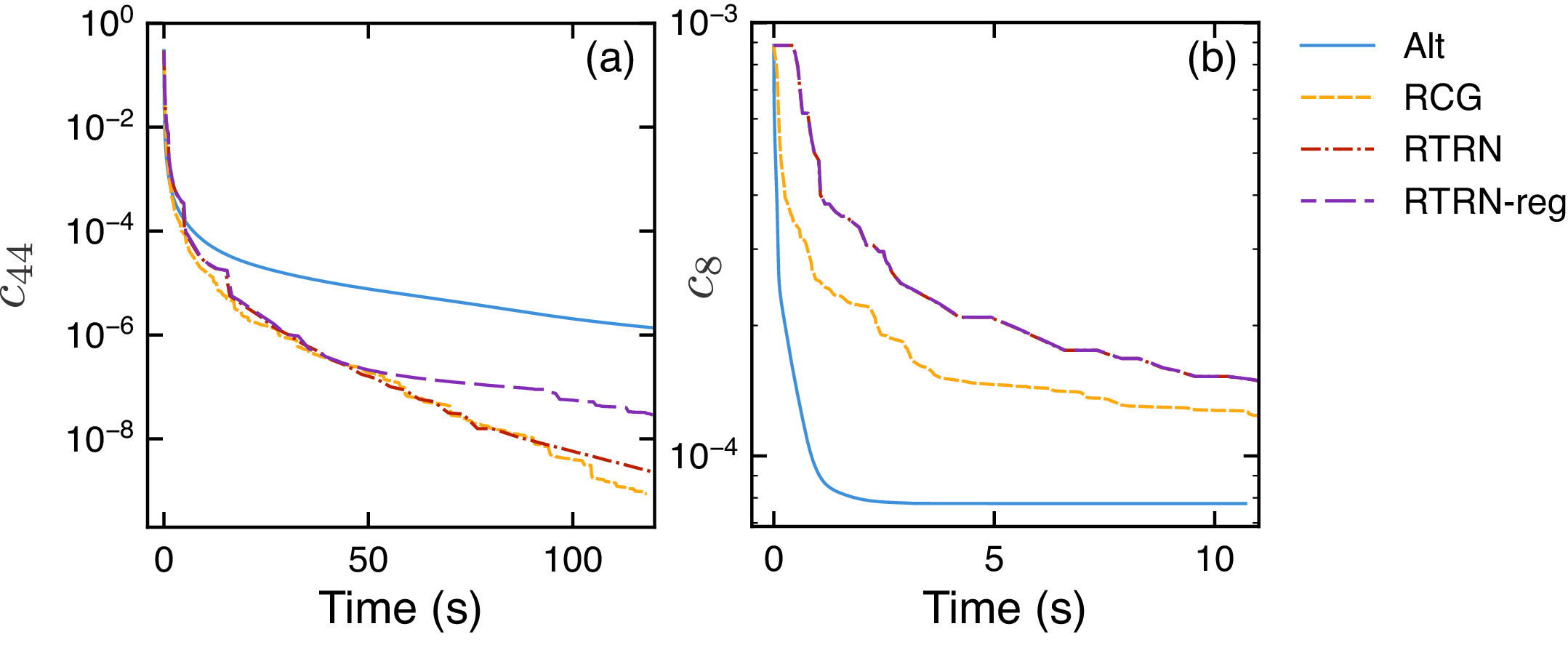}
\caption{\label{fig:methods_comp_time} Wall time of different methods for optimizing the truncation error at fixed rank $k$. (a) A random tensor of size (12, 12, 12, 12) with $k=44$. (b) An isoTNS tensor of size (8, 8, 60, 12) with $k=8$.}
\end{figure}

Although the iteration count of each algorithm shows their convergence rate, it does not fully reflect the computational cost of the algorithm, which also depends on the cost per iteration.
A single (outer) iteration of the RTRN algorithm contains many inner conjugate gradient micro-iterations (set to a maximum of 100 inner iterations), and is thus more costly than a single iteration of RCG or Alt. 

In Figure~\ref{fig:methods_comp_time}, we compare the efficiency of different algorithms by plotting the truncation error against the wall clock time, corresponding to roughly thousands of Alt and RCG iterations. In the random case, we find that for a given truncation error, RCG and RTRN converge comparably quickly; including Hessian information does not seem to improve the convergence efficiency. While the alternating disentangler performs better for the first few iterations, it seems to eventually stagnate. In contrast, in the isoTNS case, the alternating disentangler converges the fastest to the lowest truncation error. Figure~\ref{fig:methods_comp_sv}(b) shows that Alt yields a visibly faster decay in singular values when disentangling this type of tensor. For either type of tensor, we do not find an improvement in convergence when the Hessian is regularized.

\begin{figure}[!hbtp]
\centering
\includegraphics[width=\textwidth]{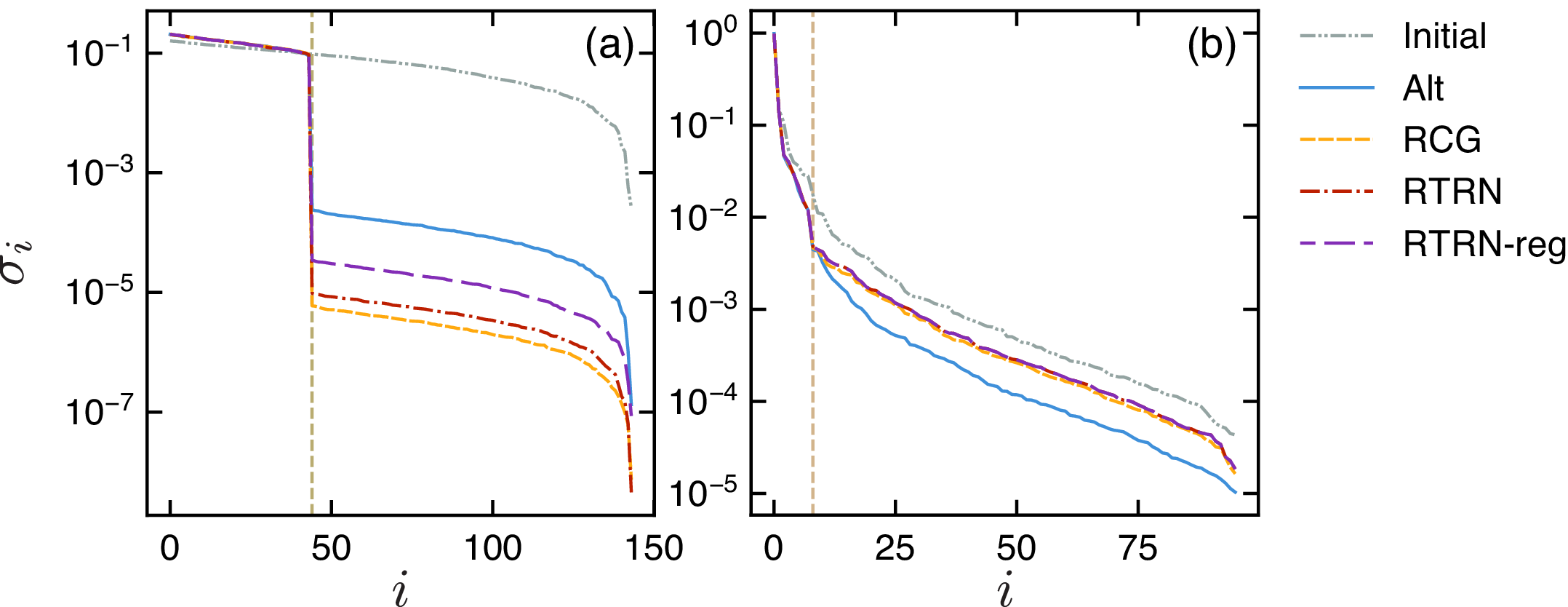}
\caption{\label{fig:methods_comp_sv} Singular value spectrum of different methods after optimizing the truncation error at fixed rank $k$, indicated by the vertical dotted lines. The spectra for different algorithms are taken at approximately the same wall time (the maximum timescales in Figure~\ref{fig:methods_comp_time}). (a) A random tensor of size (12, 12, 12, 12) with $k=44$ at time $t = 118 \rm{\ s}$. (b) An isoTNS tensor of size (8, 8, 60, 12) with $k=8$ at $t = 11 \rm{\ s}$.}
\end{figure}

\subsection{Hybrid optimization method}
We combine different optimization algorithms in order to test whether a hybrid scheme can lead to a more efficient method. For example, we can use the alternating disentangler first, and then switch to either the RCG or RTRN method as shown in Figure~\ref{fig:hybrid}. For random tensors, we observe that beginning with several iterations of the alternating disentangler before switching to RCG results in a similar convergence time as RCG, while switching to RTRN may lead to lower efficiency than RTRN. For the isoTNS tensor, where the alternating disentangler had the best performance, we observe that switching to either RCG or RTRN after starting with Alt does not improve convergence, and instead results in stagnation of the cost function. For this case, we observe that the norm of the gradient fluctuates around $10^{-4}$ or $10^{-5}$ in both RCG and RTRN (insets of Figure~\ref{fig:hybrid}(b, d) depending on the number of Alt iterations taken prior to switching to RCG or RTRN, potentially due to the presence of many local minima. Since it does not rely on the gradient, the alternating disentangler seems to escape these local minima more efficiently. These results suggest that for optimizing an unknown tensor, starting with many iterations of the Alt disentangler, followed by RCG if Alt does not converge, may be an appropriate hybrid scheme.    

\begin{figure}[!hbtp]
\centering
\includegraphics[width=\textwidth]{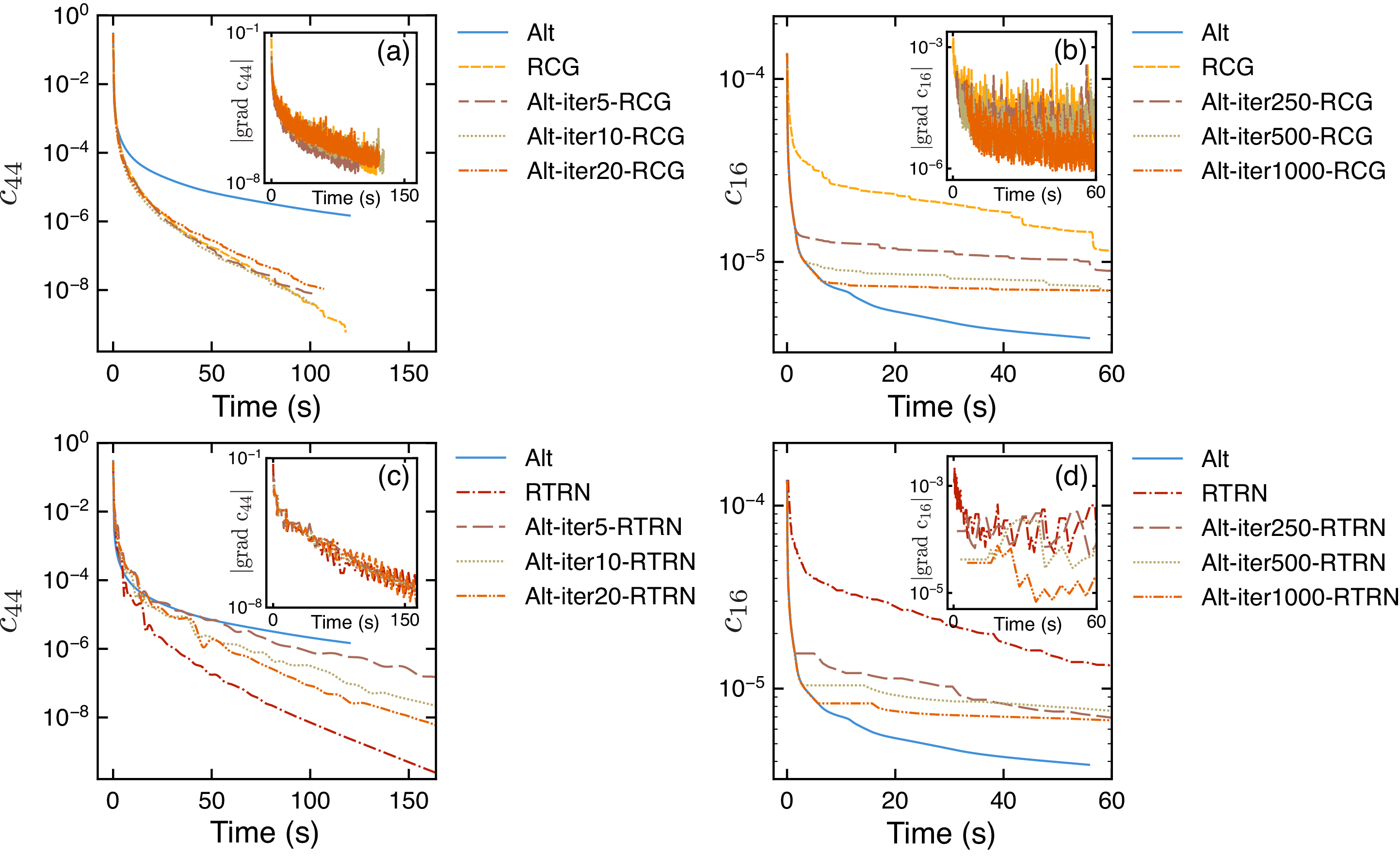}
\caption{ Wall time of different hybrid schemes for optimizing the truncation error at fixed rank $k$, where the inset shows the norm of the gradient during a representative run. (a) Timings for Alt with varied iterations (5, 10, 20), followed by RCG for a random tensor of size (12, 12, 12, 12) with $k=44$. The primary figure shows averaged timing results over five random tensors. (b) Timings for Alt with varied iterations (250, 500, 1,000), followed by RCG for an isoTNS tensor of size (8, 8, 60, 12) with $k=16$. (c) Timings for Alt with varied iterations (5, 10, 20), followed by RTRN for random tensors of the same size as (a). (d) Timings for Alt with varied iterations (250, 500, 1,000), followed by RTRN for an isoTNS tensor, as in (b). \label{fig:hybrid}}
\end{figure}

\subsection{Rank minimization} 
As we discussed earlier, the minimal rank $k_{\rm opt}$ that can be achieved by a disentangler depends on the amount of truncation error $c_{k_{\rm opt}}(Q)$ we can tolerate. The value of $k_{\rm opt}$ is generally unknown in advance.  In this section, we show how such a minimal rank can be determined for a given truncation error tolerance $\epsilon$ using the binary search method (Algorithm~\ref{alg:bisection}) presented in section~\ref{sec:practical}.

Algorithm~\ref{alg:bisection} requires an initial estimation of the minimal rank as an input. We can use either the estimate obtained from \eqref{eq:exactrankdim}, which is derived from counting the number of degrees of freedom, or perform a disentanglement optimization using one of the surrogate objective functions such as the von Neumann ($S_\textrm{vN}$) and Rényi-$\alpha$ ($S^\alpha$) entropies, defined in \eqref{eqn:vonneumann} and \eqref{eqn:renyi}, respectively. Once we obtain an optimal disentangler $Q$, we can then choose a $k$ that satisfies $c_k(Q) \leq \epsilon$.

In Figure~\ref{fig:obj_func}, we plot the truncation error \eqref{eq:sv_tail} versus truncation rank $k$ for disentanglers $Q=I$, $Q=Q_{\rm vN}$, and $Q=Q_{\mbox{R\'{e}nyi}}$, where $Q_{\rm vN}$ and $Q_{\mbox{R\'{e}nyi}}$ are obtained by minimizing the von Neumann and Rényi-1/2 entropy objective functions using the RCG algorithm. 
We can see from Figure~\ref{fig:obj_func}(a) that, for the random tensor, the truncation error of $c_k(Q_{\rm vN})$ decreases slowly with respect to $k$. The rate of change in truncation error $c_k(Q_{\rm vN})$ is similar to that of $c_k(I)$, although $c_k(Q_{\rm vN})$ is slightly below $c_k(I)$ for each $k$. This indicates that $\sigma_i({\bf A}(QX))$ decreases slowly with respect to $i$ for this choice of $Q$, which is shown in Figure~\ref{fig:obj_func_sv}(a).

When the R\'{e}yi-1/2 entropy $S^{1/2}$ is chosen as the objective function, $c_k(Q_{\mbox{R\'{e}nyi}})$ starts to decrease more rapidly when $k$ is close to 50, and becomes less than $10^{-20}$ when $k$ is slightly less than 100. 
Clearly, $Q_{\mbox{R\'{e}nyi}}$ is a much better disentangler for the random tensor.  Using the error tolerance $\epsilon = 10^{-6}$, we obtain the initial estimate $k=89$. Using this initial rank estimate as the input to the binary search algorithm, we can obtain the minimal rank $k_{\rm opt}=44$ in six steps. Table~\ref{table:binarysearch_randn} shows the change of rank and the corresponding $c_k$. We observe that this tolerance yields a rank estimate that matches Equation \eqref{eq:exactrankn}, and that the truncation error rapidly increases once $k < 44$. 
If we were to use the degrees of freedom count, which provides an initial rank estimate of $k^*=44$, we note that one fewer binary search is needed to reach the optimal rank $k_{\rm opt}$.

\begin{table}
    \centering
    \begin{tabular}{cccccc}
    \toprule
     \texttt{iter} & $k$ & $k_l$ & $k_r$ & $k_{\rm opt}$ & $c_k$ \\
    \midrule
    $1$ & $45$ & $0$ & $89$ & $89$ & $2.35 \times 10^{-11}$ \\
    $2$ & $22$ & $0$ & $44$ & $45$ & $2.39 \times 10^{-2}$ \\
    $3$ & $34$ & $23$ & $44$ & $45$ & $2.08 \times 10^{-3}$\\
    $4$ & $40$ & $35$ & $44$ & $45$ & $1.24 \times 10^{-4}$ \\
    $5$ & $43$ & $41$ & $44$ & $45$ & $1.94 \times 10^{-6}$ \\
    $6$ & $\bf{44}$ & $44$ & $44$ & $45$ & 
    $\bf{8.80 \times 10^{-10}}$ \\
    \bottomrule
    \end{tabular}
\caption{Binary search results for minimizing the rank of a random tensor of size (12, 12, 12, 12) using CG optimization, given a truncation error tolerance $\epsilon = 10^{-6}$. The final value of $k_{\rm opt}$ and the corresponding cost $c_{k_{\rm opt}}$ are highlighted in bold.}
\label{table:binarysearch_randn}
\end{table}

\begin{table}
    \centering
    \begin{tabular}{cccccc}
    \toprule
     \texttt{iter} & $k$ & $k_l$ & $k_r$ & $k_{\rm opt}$ & $c_k$ \\
    \midrule
    $1$ & $19$ & $0$ & $38$ & $38$ & $1.99 \times 10^{-6}$ \\
    $2$ & $29$ & $20$ & $38$ & $38$ & $4.44 \times 10^{-7}$ \\
    $3$ & $\bf{24}$ & $20$ & $28$ & $29$ & $\bf{8.25 \times 10^{-7}}$ \\
    $4$ & $22$ & $20$ & $23$ & $24$ & $1.23 \times 10^{-6}$ \\
    $5$ & $23$ & $23$ & $23$ & $24$ & 
    $1.04 \times 10^{-6}$ \\
    \bottomrule
    \end{tabular}
\caption{Binary search results for minimizing the rank of an isoTNS tensor of size (8, 8, 60, 12) using alternating disentangler optimization, given truncation error tolerance $\epsilon = 10^{-6}$. The final value of $k_{\rm opt}$ and the corresponding cost $c_{k_{\rm opt}}$ are highlighted in bold.}
\label{table:binarysearch_tfi}
\end{table}

In contrast, for the isoTNS example, the initial tensor already has a relatively fast decay in singular values. In Figure~\ref{fig:obj_func}(b), we instead observe that the values of $c_k(Q)$ are comparable for $Q=Q_{\rm vN}$ and $Q=Q_{\mbox{R\'{e}nyi}}$. Since $b, c \gg l,r$ in the isoTNS case, the minimal $k$ that makes \eqref{eq:exactrankdim} no less than $lrbc$ for $l=r=8$, $b=60$ and $c=12$ is $k^* = 91$. This is close to the full rank of $\bf{A}(X)$, and a poor initial guess of the minimal rank given the error tolerance $\epsilon =10^{-6}$. 
Therefore, we use the initial rank estimate obtained from the optimal $Q_{\mbox{R\'{e}nyi}}$ as the input to Algorithm~\ref{alg:bisection} to seek the minimal rank. Starting with an initial rank estimate of $k=38$, Table~\ref{table:binarysearch_tfi} shows that a few iterations are sufficient to find a minimal rank of $k_{\rm opt} = 24$. We observe in Figure \ref{fig:obj_func}(b) that compared to optimizing the Rényi-1/2 objective, using Alt to optimize (\ref{eq:opt_SVs}) with a target rank of $k_{\rm opt} = 24$ further reduces the truncation error by a factor of $6$.

\begin{figure}[!hbtp]
\centering
\includegraphics[width=\textwidth]{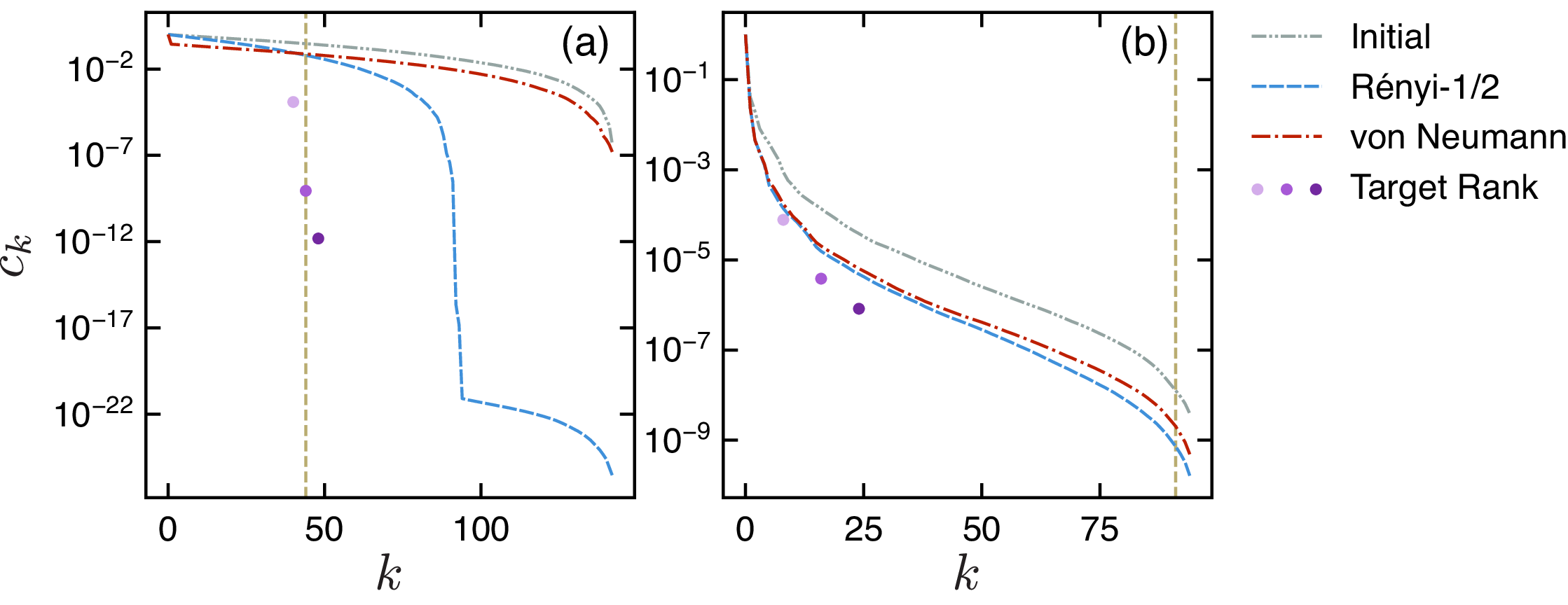}
\caption{Comparison of different objective functions for rank estimation. The von Neumann and Rényi-1/2 objectives are optimized using RCG. The vertical dotted lines indicate $k^*$ (Equation \ref{eq:exactrankdim}). (a) A random tensor of size (12, 12, 12, 12), with target rank values $k \in \{40, 44, 48\}$. The truncation error objective is optimized using RCG. (b) An isoTNS tensor of size (8, 8, 60, 12), with target rank values $k \in \{8, 16, 24\}$. The truncation error objective is optimized using Alt.}
\label{fig:obj_func}
\end{figure}

\begin{figure}[!hbtp]
\centering
\includegraphics[width=\textwidth]{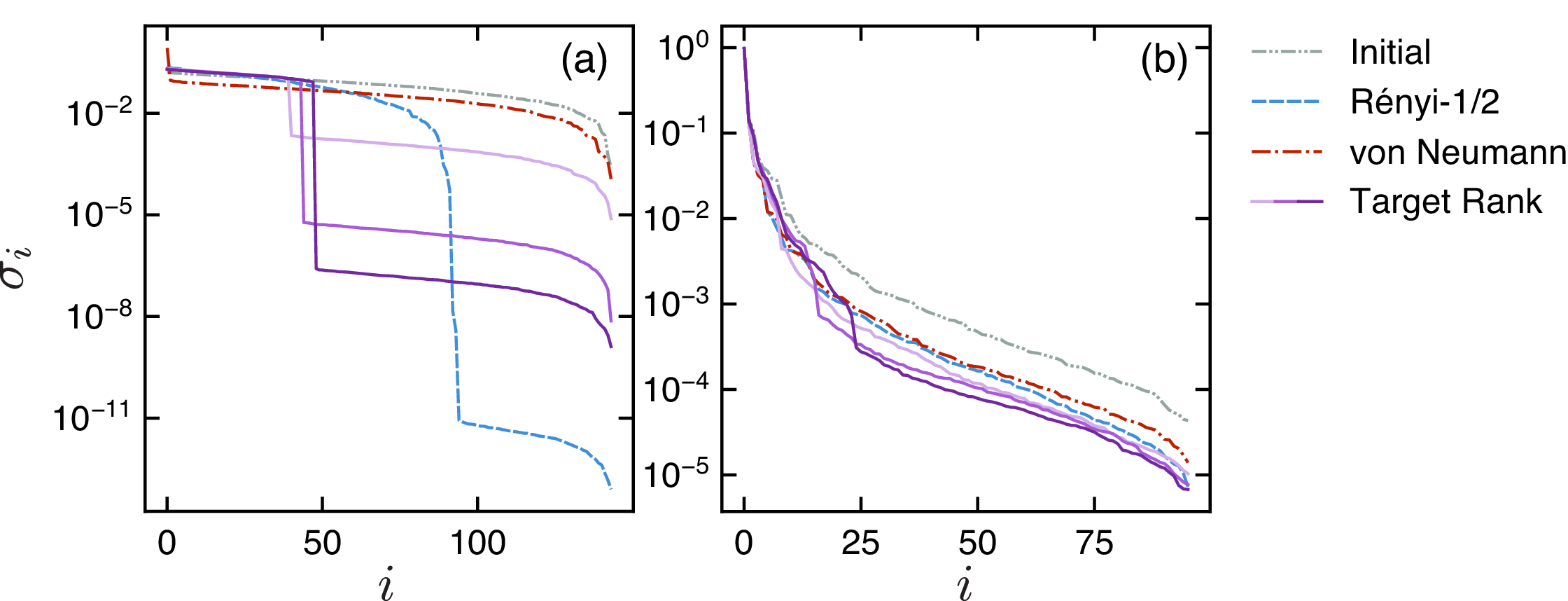}
\caption{ 
Comparison of the singular value distribution after optimization of different objective functions, using the same methods as Figure~\ref{fig:obj_func}. (a) A random tensor of size (12, 12, 12, 12). (b) An isoTNS tensor of size (8, 8, 60, 12).}
\label{fig:obj_func_sv}
\end{figure}

\section{Conclusion}
\label{sec:conclude}
In this work, we have formulated the tensor disentanglement problem within the framework of Riemannian optimization. We discussed the choice of different objective functions and showed how both the Riemannian gradient and Hessian-vector products can be computed efficiently for first- and second-order Riemannian optimization algorithms. We have also shown that the disentangling task can be formulated as a joint optimization over orthogonal and low-rank matrices, and solved by an alternating minimization algorithm with lower per-iteration complexity compared to Riemannian optimization methods. Practical issues such as choosing an initial rank estimate and minimizing the rank via a bisection algorithm were also addressed. 

Numerical experiments were conducted on both synthetic random and application-driven tensors. The results show that while the Riemannian trust-region Newton method achieves the fastest convergence in terms of the number of outer iterations, its high per-iteration cost limits its practical utility. In contrast, the Riemannian conjugate gradient algorithm is more efficient for random tensors, whereas alternating minimization is the fastest for structured tensors arising from physics applications. Furthermore, we demonstrated how the bisection strategy can effectively identify the minimal rank, with different rank initialization strategies.

Looking forward, we aim to improve the practical viability of second-order methods by developing effective preconditioners to accelerate the solution of the Newton correction equation. Additionally, we plan to investigate alternative constraints on the disentangler in order to enhance the quality of entanglement reduction across the tensor network and broaden the applicability of the proposed framework.

\bmhead{Acknowledgements}
The authors wish to thank S. Anand, S.-H. Lin, Y. Wu, M. P. Zaletel, N. Y. Yao, C. Zhang, and A. Seigal for helpful discussions about running isoTNS algorithms and the existence of exact disentanglers. J. W. acknowledges support by the U.S. Department of Energy Computational Science Graduate Fellowship under award number DE-SC0022158. This work was also supported in part by the U.S. Department of Energy, Office of Science, Office of Advanced Scientific Computing Research, Scientific Discovery through Advanced Computing (SciDAC) Program through the FASTMath Institute under U.S. Department of Energy Contract No. DE-AC02-05CH11231. (A. D. and C. Y.)

\bibliography{sources}
\end{document}